\documentclass[12pt,twoside]{article}

\usepackage{graphicx,graphics, amsthm, amsfonts, amsbsy,amssymb, amsmath, cite, array,arydshln, hyperref, float,tikz}
\usepackage[utf8]{inputenc}

\let\OLDthebibliography\thebibliography
\renewcommand\thebibliography[1]{
	\OLDthebibliography{#1}
	\setlength{\parskip}{3pt}
	\setlength{\itemsep}{0pt plus 0.3ex}
}

\usepackage[small,compact]{titlesec} 
\everymath{\displaystyle}

\usepackage[margin=1in]{geometry}

\allowdisplaybreaks

\usepackage{enumitem}
\newtheorem{theorem}{Theorem}[section]
\newtheorem{remark}[theorem]{Remark}
\newtheorem{example}[theorem]{Example}
\newtheorem{problem}{Problem}
\newtheorem{lemma}[theorem]{Lemma}

\newtheorem{corollary}[theorem]{Corollary}

\newtheorem{proposition}[theorem]{Proposition}

\renewenvironment{proof}{{\noindent\bfseries Proof.}}{\qed}

\title{Explicit Formulas and Unimodality Phenomena for General Position Polynomials}
\author{
	{ Bilal Ahmad Rather$^{a}$} \\[2mm]
	\small $^{a}$School of Mathematics and Statistics, Shandong University of Technology, Zibo 255049, China\\
	$^{a}$\texttt{bilalahmadrr@gmail.com}
}
\date{}

\begin{document}
	\pagestyle{myheadings} \markboth{Bilal Ahmad Rather}{Explicit Formulas and Unimodality Phenomena for General Position Polynomials}
\maketitle

	\begin{abstract}
		The general position problem in graphs seeks the largest set of vertices such that no three vertices lie on a common geodesic. Its counting refinement, the general position polynomial $\psi(G)$, asks for all such possible sets. In this paper, We describe general position sets for several classes of graphs and provide explicit formulas for the general position polynomials of complete multipartite graphs. We specialize to balanced complete multipartite graphs and show that for part size $r\le 4$, the polynomial $\psi(K_{r,\dots,r})$ is log-concave and unimodal for all numbers of parts, while for larger $r$, counterexamples show that these properties fail. Finally, we analyze the corona $G\circ K_1$ and prove that unimodality of $\psi(G)$ is retained for some classes, and counterexamples exists for complete bipartite and complete multipartite graphs. The results verify the analogy between general position polynomials and classical position-type parameters, and establish balanced multipartite graphs and coronas as potential subjects for further investigation.
	\end{abstract}

	\vskip 3mm
	
	
\section{Introduction and motivation}

All graphs considered are finite, simple and undirected. For a graph $G$, we denote its vertex set by $V(G)$ and its order by $|V(G)|$. The distance between vertices $u,v\in V(G)$ is denoted by $d_G(u,v)$ or simply $d(u,v)$. Two adjacent vertices are represented by $u\sim v.$
The \emph{general position problem} was introduced independently by Chandran and Parthasarathy \cite{7} and by Manuel and Klav\v{z}ar \cite{20}. For a connected graph $G$, a vertex set $S\subseteq V(G)$ is in \emph{general position} if no vertex of $S$ lies on a shortest path between two other vertices of $S$, that is, there are no distinct $x,y,z\in S$ with \[
d(x,z)=d(x,y)+d(y,z).
\]
The \emph{general position number} $\operatorname{gp}(G)$ is the maximum
cardinality of a general position set. This parameter has been investigated on
many graph classes and products, including cographs and bipartite graphs
\cite{4}, Cartesian products \cite{17,18,23}, various interconnection networks
\cite{21}, Kneser graphs \cite{10,22}, and graphs with small diameter or other
structural constraints \cite{20,21,26}. Algorithmic and game-theoretic aspects
have been studied in \cite{15,16}, while related position-type notions such as
mutual visibility and monophonic position sets have been considered in
\cite{9,24,25}. 

A finite sequence of nonnegative real numbers \((\beta_0,\dots,\beta_d)\) is \emph{unimodal} if there exists an index \(m\) such that
\[
\beta_0\le \beta_1\le \dots\le \beta_m\ge \beta_{m+1}\ge \dots\ge \beta_d,
\]
and it is \emph{log-concave} if
\[
\beta_k^2 \ge \beta_{k-1}\beta_{k+1}
\]
for all \(1\le k\le d-1\). A log-concave sequence with no internal zeros (no zero between two non-zero entries) is automatically unimodal, see \cite{1,8,11,13,14} for discussions in related settings.

As with many extremal graph parameters, it is natural to refine
$\operatorname{gp}(G)$ by counting all general position sets. This leads to
the \emph{general position polynomial} (shortly GPA) \cite{0}. For a graph $G$, let $\alpha_k(G)$ denote
the number of general position sets of size $k$. Then, the GPA is defined as
\[
\psi(G)=\sum_{k\ge 0} \alpha_k(G) x^k.
\]
It is the ordinary generating polynomial of general position sets. 
The first systematic study of $\psi(G)$ was carried out in \cite{0}. The closed formulas were obtained for $\psi(G)$ on a range of graph
families  like paths, cycles, complete graphs, complete bipartite graphs, thin and
thick grids, combs, Kneser graphs $K(n,2)$, and for several graph operations
(disjoint union, join, Cartesian product). It was also shown that the general
position polynomial is not unimodal in general. The explicit non-unimodal
examples arise already among trees (such as suitable brooms) and among
complete bipartite graphs, for instance $K_{8,4}$ and $K_{9,7}$.

The study of the polynomial $\psi(G)$ fits into a broader line of research on graph polynomials whose coefficients encode combinatorial information. It behaviour similar like the classical polynomials, some examples include: the independence polynomial, whose coefficients count independent sets,  the matching polynomial, whose real-rootedness was proved in \cite{11},   the chromatic polynomial, whose coefficients are unimodal \cite{13},   the domination polynomial, introduced in \cite{3}, for which unimodality has been conjectured \cite{3} and partially verified in	\cite{2,5,6,19}, and the clique and independent set polynomials studied in \cite{12}. The coefficient sequences of these polynomials are often conjectured or proved to be \emph{unimodal} or \emph{log-concave}. In particular, the independence polynomial of a claw-free graph is real-rooted and hence unimodal \cite{8}, the
matching polynomial is real-rooted \cite{11}, and the chromatic polynomial has
unimodal coefficients \cite{13}. On the other hand, independence polynomials
are not unimodal in general, although it was conjectured in \cite{1} that they
are unimodal on trees. Recent work shows that even for trees, log-concavity may fail \cite{14}.

In contrast to these classical polynomials, the GPA
$\psi(G)$ encodes a genuinely geometric constraint—avoidance of geodesics.
Ir\v{s}i\v{c} et al.\ \cite{0} showed that $\psi(G)$ displays a rich mixture of
behaviours: for certain classes (paths, cycles, combs, Kneser graphs $K(n,2)$)
it is unimodal or even has additional structure, while for others (brooms,
some complete bipartite graphs) it fails to be unimodal.

\medskip
\noindent\textbf{Motivations and scope of this work.}
With the motivation of above classical polynomials, studying the unimodality and log-concavity of $\psi(G)$ serves several purposes, and we have following reason to carry forward the study of polynomial $\psi(G)$.
\begin{enumerate}
	\item It clarifies the distribution of general position sets by size and
	reveals the shape of position-type configurations within a graph.
	\item It links geometric constraints on shortest paths with algebraic
	properties of generating functions.
	\item It connects and extends several lines of work on general position,
	visibility, and extremal position problems in graphs
	\cite{4,7,9,15,16,20,21,23,24,25,26}.
	\item It dovetails with long-standing themes and conjectures on unimodality
	and real-rootedness of graph polynomials
	\cite{1,2,3,5,6,8,11,12,13,14,19}.
\end{enumerate}

In this paper, we focus on two broad and natural directions. First, we study  complete multipartite graphs, which unify complete graphs, complete bipartite graphs, Turán graphs and their relatives. They may be viewed as joins of independent sets, and joins are among the basic graph operations for which general position numbers and polynomials have been investigated \cite{0,10,17,21,23,bilalsc,bilaltcs}. We provide a structural description of general position sets in complete multipartite graphs, leading directly to closed formulas for their GPAs. As consequences,  we recover the known formula for complete bipartite graphs and clarify how non-unimodality can already appear in this setting. Second, we specialize to  balanced complete multipartite graphs, that is, complete $a$-partite graphs with all parts of the same size $r$. For such graphs, we analyze the coefficient sequence of $\psi(K_{r,\dots,r})$ in detail. Our main contributions here are: for small part size $r\le 4$, we show that $\psi(K_{r,\dots,r})$ is 	log-concave and hence unimodal, for all numbers of parts $a$, and  for larger $r$ we construct explicit counterexamples (including 	balanced complete multipartite graphs) where the general position
polynomial fails to be unimodal and where log-concavity also breaks down. 
This identifies a natural threshold phenomenon for unimodality and log-concavity in balanced complete multipartite graphs.

Motivated by the role of graph operations in the general position problem
\cite{0,10,17,20,21,23}, we also examine the \emph{corona} $G\circ K_1$,
obtained from $G$ by attaching a pendent vertex to each vertex of $G$. We show
that unimodality of $\psi(G)$ is preserved for several basic families (paths,
edgeless graphs, combs), but the general question whether $\psi(G)$ unimodal
implies $\psi(G\circ K_1)$ unimodal remains open.

\medskip
 In Section~\ref{sec:prelims} we recall basic definitions and simple facts about distances and shortest paths in complete multipartite graphs. In Section~\ref{sec:polynomial} we derive a closed formula for the general
position polynomial of complete multipartite graphs and recover the complete
bipartite case as a special instance.  In Section~\ref{sec:balanced} we
specialize to balanced complete multipartite graphs, study the coefficient
sequence of $\psi(K_{r,\dots,r})$, and establish log-concavity and unimodality
when $r\le 4$, together with explicit counterexamples beyond this range.  In Section \ref{sec 5}, we disuss the log-concave and unimodal properties of the GPA of some graphs. In
Section~\ref{section 6},  we consider the corona operation $G\circ K_1$ and
its effect on unimodality of the GPA. We conclude in Section~\ref{section 7} with log-concave for the corona operation $G\circ K_1$. We end up article with some comments for future direction in Section \ref{section 8}.	
	\section{Preliminaries}\label{sec:prelims}
	
	We begin with the basic definitions needed in the sequel. Fro a graph \(G\), A subset \(S\subseteq V(G)\) is a \emph{general position set} if there do not exist distinct vertices \(x,y,z\in S\) such that
		\[
		d(x,z)=d(x,y)+d(y,z),
		\]
		that is, no vertex of \(S\) lies on a shortest path between two other vertices of \(S\). The \emph{general position number} of \(G\) is
		\[
		\operatorname{gp}(G)=\max\{|S|: S\subseteq V(G) \text{ is a general position set}\}.
		\]
		For each \(k\ge 0\), let \(\alpha_k(G)\) denote the number of general position sets of size \(k\). The \emph{general position polynomial} of \(G\) is
		\[
		\psi(G)=\sum_{k\ge 0} \alpha_k(G) x^k.
		\]
	We note that  \(\alpha_0(G)=1\), \(\alpha_1(G)=|V(G)|\), and every pair of vertices forms a general position set, so \(\alpha_2(G)=\binom{|V(G)|}{2}\). Thus, for a graph of order \(n\),
	\[
	\psi(G)=1+nx+\binom{n}{2}x^2+\cdots,
	\]
	and the degree of \(\psi(G)\) equals \(\operatorname{gp}(G)\).
	
	A graph \(G\) is \emph{complete multipartite} if its vertex set can be partitioned into independent sets 	$V_1,\dots,V_t$,  called the \emph{partite sets} or simply \emph{parts}, such that two vertices are adjacent if and only if they belong to different parts. If \(|V_i|=n_i\) for each \(i\), we denote such a graph by $K_{n_1,\dots,n_t}.$  Thus, complete graphs, complete bipartite graphs and Turán graphs are all special cases of complete multipartite graphs.
	
	The following basic observations will be used repeatedly.
	 \begin{lemma}\label{lem:basic-small}
		Let \(G\) be a connected graph of order \(n\).
		\begin{enumerate}
			\item Every subset of \(V(G)\) of size at most \(2\) is in general position. In particular,
			 $\alpha_0(G)=1, \alpha_1(G)=n,$ and $ \alpha_2(G)=\binom{n}{2}.$ 
			\item If \(H\) is an isometric induced subgraph of \(G\) and \(S\subseteq V(H)\), then \(S\) is a general position set of \(H\) if and only if it is a general position set of \(G\).
		\end{enumerate}
	\end{lemma}
	\begin{proof}
		Part (1) is immediate, since the definition involves three distinct vertices.
		
		For (2), let \(x,y,z\in S\) be distinct. Since \(H\) is isometric in \(G\), so
		\[
		d_H(x,y)=d_G(x,y),\quad d_H(y,z)=d_G(y,z),\quad d_H(x,z)=d_G(x,z).
		\]
		Therefore, we have
		\[
		d_H(x,z)=d_H(x,y)+d_H(y,z)
		\iff
		d_G(x,z)=d_G(x,y)+d_G(y,z).
		\]
		Hence no vertex of \(S\) lies on a shortest path between two others in \(H\) if and only if the same holds in \(G\).
	\end{proof}
	
	The following result records the distances in complete multipartite graphs.
	\begin{lemma}\label{lem:distances}
		Let \(G=K_{n_1,\dots,n_t}\) with \(t\ge 2\).
		\begin{enumerate}
			\item If \(u,v\) lie in different parts, then \(d(u,v)=1\).
			\item If \(u\neq v\) lie in the same part, then \(d(u,v)=2\).
		\end{enumerate}
	\end{lemma}
	
	\begin{proof}
		If \(u,v\) lie in different parts, they are adjacent by definition, so \(d(u,v)=1\). If \(u,v\) lie in the same part, they are non-adjacent, However, since \(t\ge 2\) there exists a vertex \(w\) in a different part, and then \(u\sim w\sim v\), so \(d(u,v)\le 2\). Hence, we obtain \(d(u,v)=2\).
	\end{proof}
	
\section{General position polynomials of complete multipartite graphs}\label{sec:polynomial}
	
	In this section, we characterize general position sets in complete multipartite graphs. The description is remarkably simple: a general position set is either contained in a single part or contains at most one vertex from each part.
	
	\begin{theorem}\label{thm:structure-gpsets}
		Let \(G=K_{n_1,\dots,n_t}\) with \(t\ge 2\), and let \(V_1,\dots,V_t\) be its partite sets. A subset \(S\subseteq V(G)\) is a general position set if and only if one of the following holds:
		\begin{enumerate}
			\item[(A)] \(S\) is contained in a single part, that is, \(S\subseteq V_i\) for some \(i\in[t]\), or
			\item[(B)] \(S\) contains at most one vertex from each part, that is,  \(|S\cap V_i|\le 1\) for all \(i\in[t]\).
		\end{enumerate}
	\end{theorem}
	
	\begin{proof}
		First suppose \(S\) satisfies (A) or (B). If \(S\subseteq V_i\), let \(u,v\in S\) with \(u\neq v\). By Lemma~\ref{lem:distances} (2), we have \(d(u,v)=2\), and any shortest \(u\text{-}v\) path has the form \(u\sim w\sim v\), where \(w\) lies in some \(V_j\) with \(j\neq i\). Since, all vertices of \(S\) lie in \(V_i\), no vertex of \(S\) can be an internal vertex of such a path. Thus, no vertex of \(S\) lies on a shortest path between two others, and \(S\) is in general position.
		
		If (B) holds, then any two vertices \(u,v\in S\) lie in different parts, so by Lemma~\ref{lem:distances} (1), we obtain \(d(u,v)=1\). A shortest \(u\text{-}v\) path is of length \(1\), and contains no internal vertex. In particular no third vertex of \(S\) lies on such a path. Hence, \(S\) is in general position.
		
		Conversely, assume that \(S\) is a general position set, and does not satisfy (A), that is, it meets at least two parts. Suppose that (B) fails, then there exists a part, say \(V_i\), such that \(|S\cap V_i|\ge 2\), and another part, say \(V_j\) with \(j\neq i\), such that \(S\cap V_j\neq\emptyset\). Select unique vertices \(u,v\in S\cap V_i\), and a vertex \(w\in S\cap V_j\). According to Lemma~\ref{lem:distances} (2), \(d(u,v)=2\), and by Lemma~\ref{lem:distances} (1), \(d(u,w)=d(v,w)=1\). Consequently, it implies that that $d(u,v) = 2 = d(u,w) + d(w,v).$ Thus, \(w\) is situated on a minimal \(u\text{-}v\) path. This contradicts the assumption that \(S\) constitutes a set in general position. Hence, if \(S\) is a general position set not contained in a single part, it must satisfy (B).
	\end{proof}
	
	As an immediate consequence, we obtain the general position number.
	
	\begin{corollary}\label{cor:gp-number-multipartite}
		Let \(G=K_{n_1,\dots,n_t}\) with \(t\ge 2\), and let $N=\sum_{i=1}^t n_i,$ and $M=\max_{1\le i\le t} n_i.$ 	Then
		$\operatorname{gp}(G)=\max\{M,t\}.$ 
	\end{corollary}
	
	\begin{proof}
		By Theorem~\ref{thm:structure-gpsets}, a general position set is either contained in a single part, in this case its size is at most \(M\), or uses at most one vertex from each of the \(t\) parts, in this case its size is at most \(t\). Any largest part \(V_i\) provides a general position set of size \(M\), and any choice of one vertex from each part yields a general position set of size \(t\). Hence, we have \(\operatorname{gp}(G)=\max\{M,t\}\).
	\end{proof}
	
	\begin{remark}
		For complete graphs \(K_n\) we may regard the graph as \(K_{1,\dots,1}\) with \(t=n\). Then \(M=1\) and \(t=n\), so Corollary~\ref{cor:gp-number-multipartite} yields \(\operatorname{gp}(K_n)=n\), as expected. For complete bipartite graphs \(K_{m,n}\) we have \(t=2\) and \(M=\max\{m,n\}\), so \(\operatorname{gp}(K_{m,n})=\max\{m,n\}\), recovering the known result \cite{10,20}.
	\end{remark}

	We now derive a closed formula for the GPA of \(K_{n_1,\dots,n_t}\). Let $G = K_{n_1,\dots,n_t},$ with $V_1,\dots,V_t $  as before $|V_i|=n_i,$ for  $1\leq i\leq t$. Let \(N = \sum_{i=1}^t n_i\) be the order of \(G\).  By Theorem~\ref{thm:structure-gpsets}, every general position set is of one of the following two types: (A) sets entirely contained in a single part, and (B) sets that use at most one vertex from each part. We count the number of general position sets of each type and size. For \(k\ge 0\), let \(\alpha_k(G)\) denote the number of general position sets of size \(k\). A general position set of type (B) of size \(k\) is obtained by choosing a subset \(I\subseteq [t]\) of parts with \(|I|=k\) or choosing one vertex from each \(V_i\) with \(i\in I\). Hence, for \(k\ge 0\), we have
	\begin{equation}\label{eq:typeB}
		\alpha_k^{(B)}
		=\sum_{\substack{I\subseteq [t]\\ |I|=k}} \prod_{i\in I} n_i.
	\end{equation}
	The right-hand side is the elementary symmetric polynomial
	\[
	e_k(n_1,\dots,n_t)
	= \sum_{\substack{I\subseteq [t]\\ |I|=k}} \prod_{i\in I} n_i,
	\]
	with the convention that \(e_0(n_1,\dots,n_t)=1\). A general position set of type (A) of size \(k\) is obtained by choosing a part \(V_i\) and then choosing a \(k\)-subset of \(V_i\). Thus, we have
	\begin{equation}\label{eq:typeA}
		\alpha_k^{(A)}
		=\sum_{i=1}^t \binom{n_i}{k}.
	\end{equation}
	Note that for \(k\ge 2\), the two types are disjoint: a set of size at least 2 contained in a single part uses two vertices from that part and hence cannot satisfy the at most one per part condition. For \(k=0\) and \(k=1\) there is overlap, but we can easily separate these cases.
	
	We can now state and the main formula related to  GPA of multipartite graph.
		\begin{theorem}\label{thm:psi-multipartite}
		Let \(G=K_{n_1,\dots,n_t}\) with \(t\ge 2\) and total order \(N=\sum_{i=1}^t n_i\). Let \(e_k(n_1,\dots,n_t)\) be the elementary symmetric polynomials in the variables \(n_1,\dots,n_t\). Then the GPA of \(G\) is
		\begin{equation}\label{eq:psi-multipartite}
			\psi(G)
			= 1 + N x + \sum_{k=2}^{d} \left( e_k(n_1,\dots,n_t) + \sum_{i=1}^t \binom{n_i}{k} \right)x^k,
		\end{equation}
		where $d = \operatorname{gp}(G) 
		= \max\Bigl\{ t,\ \max_{1\le i\le t} n_i \Bigr\}.$ Equivalently, $\alpha_0(G)=1, \alpha_1(G)=N,$ and, for \(k\ge 2\),
		\begin{equation}\label{eq:coeff-general}
			\alpha_k(G)
			= e_k(n_1,\dots,n_t) + \sum_{i=1}^t \binom{n_i}{k},
		\end{equation}
		with the understanding that \(e_k(n_1,\dots,n_t)=0\), if \(k>t\), and \(\binom{n_i}{k}=0\), if \(k>n_i\).
	\end{theorem}
	
	\begin{proof}
		By Lemma~\ref{lem:basic-small}, \(\alpha_0(G)=1\) and \(\alpha_1(G)=N\). For each \(k\ge 2\), every general position set of size \(k\) is either of type (A) or type (B), and the two types are disjoint. Thus $\alpha_k(G)=\alpha_k^{(A)}+\alpha_k^{(B)}$ 	for \(k\ge 2\). Using \eqref{eq:typeB} and \eqref{eq:typeA} we have \eqref{eq:coeff-general} as
		\[
		\alpha_k(G)=e_k(n_1,\dots,n_t)+\sum_{i=1}^t \binom{n_i}{k},
		\]
		 The degree of \(\psi(G)\) equals \(\operatorname{gp}(G)\), which is given by Corollary~\ref{cor:gp-number-multipartite}, hence the sum in \eqref{eq:psi-multipartite} runs up to \(d=\operatorname{gp}(G)\).
	\end{proof}
	
	 We  illustrate Theorem~\ref{thm:psi-multipartite} on a few standard examples. A complete graph \(K_n\) is viewed as a complete multipartite graph with \(t=n\) parts of size $1$. Thus, for \(k\ge 2\), $\sum_{i=1}^n \binom{n_i}{k} = 0,$ while
		$e_k(1,\dots,1)=\binom{n}{k}.$ So, we obtain
		\[
		\psi(K_n)=\sum_{k=0}^n \binom{n}{k}x^k = (1+x)^n,
		\]
		and is in accordance with \cite[Proposition 2.1(i)]{0}.
	
	Let \(G=K_{m,n}\) with \(m\ge n\). Then \(t=2\), \(n_1=m\), \(n_2=n\). For \(k\ge 2\),
		\[
		e_k(m,n)
		=
		\begin{cases}
			mn; & k=2,\\
			0; & k\ge 3,
		\end{cases}
		\]
		and $\sum_{i=1}^2 \binom{n_i}{k}= \binom{m}{k}+\binom{n}{k}.$ Thus, we have $\alpha_0(G)=1, \alpha_1(G)=m+n,
		\alpha_2(G)=mn+\binom{m}{2}+\binom{n}{2}=\binom{m+n}{2},$ 
		and for \(k\ge 3\), $\alpha_k(G)=\binom{m}{k}+\binom{n}{k}.$ 
		Therefore, the GPA of $K_{m,n}$ is
		\[
		\psi(K_{m,n})
		= 1 + (m+n)x + \binom{m+n}{2}x^2
		+ \sum_{k=3}^{m}\left(\binom{m}{k}+\binom{n}{k}\right)x^k,
		\]
		which coincides with the expression for \(\psi(K_{m,n})\) previously obtained in \cite[Proposition 2.1(v)]{0}. 
	
	\begin{example}
		Let \(G=K_{r,\dots,r}\) be a complete multipartite graph including \(t\) partitions, each of equal size \(r\ge 1\).Then \(n_i=r\) for all \(i\in[t]\), and for each \(k\), we have $e_k(r,\dots,r)	= \binom{t}{k}r^k$, and  $\sum_{i=1}^t \binom{n_i}{k} = t\binom{r}{k}.$	Thus for \(k\ge 2\),
		\begin{equation}\label{eq:ak-balanced}
			\alpha_k(G)
			= \binom{t}{k}r^k + t\binom{r}{k},
		\end{equation}
		and
		\begin{equation}\label{eq:psi-balanced}
			\psi(K_{r,\dots,r})
			= 1 + trx + \sum_{k=2}^{\max\{r,t\}}
			\left(\binom{t}{k}r^k + t\binom{r}{k}\right)x^k.
		\end{equation}
		We will consider these graphs again in next Section~\ref{sec:balanced}.
	\end{example}
	
	\section{Balanced complete multipartite graphs: unimodality and log-concavity}
	\label{sec:balanced}
	
	In this section, we investigate the coefficient sequence of \(\psi(K_{r,\dots,r})\) when all parts have equal size \(r\).  For balanced complete multipartite graphs, we have the explicit formula \eqref{eq:ak-balanced} with $\alpha_0=1, \alpha_1=tr, $ and $\alpha_k = \binom{t}{k}r^k + t\binom{r}{k}$ for $k\ge 2.$ We first show that for small part size \(r\le 4\), the GPA of \(K_{r,\dots,r}\) is log-concave and therefore unimodal, independently of the number of parts. 

	\begin{theorem}\label{thm:balanced-small-r}
		Let $t\ge 2$ and let $K_{r,\dots,r}$ be a balanced complete $t$-partite graph
		with part size $r\in\{1,2,3,4\}$. Let $\psi(K_{r,\dots,r})=\sum_{k=0}^d \alpha_k x^k,$ with $d=\max\{r,t\},$ and suppose the coefficients are given by $\alpha_0=1,\alpha_1=tr, $ and $\alpha_k = \binom{t}{k}r^k + t\binom{r}{k}$ for $k\ge 2.$ Then the sequence $(\alpha_k)_{k=0}^d$ is log-concave and hence unimodal.
	\end{theorem}
	
	\begin{proof} For $r\in\{1,2,3,4\}$ and $t\ge 2$ we first note that
		$\alpha_k>0$ for $0\le k\le d$ (and $\alpha_k=0$ for $k>d$). Hence there are no internal
		zeros in $(\alpha_k)_{k=0}^d$, and it suffices to prove log-concavity \cite{stanley}.
		We consider each case of $r$ separately. We need to recall basic facts.  For each fixed $t$, the binomial sequence $\left (\binom{t}{k}\right )_{k\ge 0}$
		is known to be log-concave in $k$, equivalently,
		\[
		\binom{t}{k}^2  \ge  \binom{t}{k-1}\binom{t}{k+1},\quad\text{for all }k.
		\]
		Multiplying term wise by $r^{2k}$, we see that the sequence
		$\left (r^k\binom{t}{k}\right )_{k\ge 0}$ is also log-concave
		\[
		\bigl(r^k\binom{t}{k}\bigr)^2
		 \ge  r^{k-1}\binom{t}{k-1}\cdot r^{k+1}\binom{t}{k+1}.
		\]
		In particular, whenever $\alpha_{k-1},\alpha_k,\alpha_{k+1}$ are all of the form
		$r^j\binom{t}{j}$, the log-concavity inequalities hold trivially.
		For $r=1$,  $K_{1,\dots,1}\cong K_t$ is a complete graph and we have (\cite[Proposition~2.1(i)]{0}), we have $\psi(K_t) = (1+x)^t.$ Thus, $\alpha_k=\binom{t}{k}$, and the sequence $(\alpha_k)$ is strictly log-concave and unimodal.
		
		For $r=2$, and from \eqref{eq:ak-balanced} we obtain $\alpha_0=1, \alpha_1=2t,$ and, for $k\ge 2$, we have $\alpha_k = \binom{t}{k}2^k + t\binom{2}{k}.$ 	In particular, $\alpha_2 = 4\binom{t}{2} + t = 2t(t-1)+t = t(2t-1),$ while for $k\ge 3$, we simply have $\alpha_k = 2^k\binom{t}{k}$. For $k\ge 4$ the triple $(\alpha_{k-1},\alpha_k,\alpha_{k+1})$ lies entirely in the
		binomial regime, so log-concavity holds by the basic facts. Thus, it
		remains to check the indices $k=1,2,3$. For $k=1$, we have $\alpha_1^2 - \alpha_0\alpha_2 = (2t)^2 - t(2t-1) = t(2t+1)  \ge  0,$
		which holds for all $t\ge 1$. For $k=2$,  $\alpha_1=2t$, $\alpha_2=t(2t-1)$ and
		$\alpha_3 = 8\binom{t}{3} = \tfrac{4}{3}t(t-1)(t-2)$, we obtain
		\begin{align*}
			\alpha_2^2 - \alpha_1\alpha_3
			&= t^2(2t-1)^2 - (2t)\cdot \frac{4}{3}t(t-1)(t-2)= \frac{t^2}{3}\bigl(4t^2 + 12t - 13\bigr).
		\end{align*}
		For \( t \ge 2 \), it follows that \( 4t^2 + 12t - 13 \ge 4t^2 + 12t - 16 = 4(t+4)(t-1) \ge 0 \), and the inequality \( \alpha_2^2 \ge \alpha_1\alpha_3 \) is satisfied.
		
		For $k=3$,  $\alpha_2=t(2t-1)$, $\alpha_3=\tfrac{4}{3}t(t-1)(t-2)$ and $\alpha_4 = 16\binom{t}{4} = \tfrac{2}{3}t(t-1)(t-2)(t-3)$. So, we have
		\begin{align*}
			\alpha_3^2 - \alpha_2\alpha_4
			&= \frac{16}{9}t^2(t-1)^2(t-2)^2
			- t(2t-1)\cdot \frac{2}{3}t(t-1)(t-2)(t-3)\\[1mm]
			&= \frac{2t^2(t-1)(t-2)}{9}\Bigl(8(t-1)(t-2) - 3(2t-1)(t-3)\Bigr)\\
			&= \frac{2t^2(t-1)(t-2)}{9}\bigl(2t^2 - 3t + 7\bigr).
		\end{align*}
		The quadratic \(2t^2 - 3t + 7\) possesses a negative discriminant (\(\Delta = 9 - 56 < 0\)) and a positive leading coefficient; therefore, it is positive for all real \(t\).  Since, also $t^2(t-1)(t-2)\ge 0,$ for $t\ge 2$, so we obtain
		$\alpha_3^2\ge \alpha_2\alpha_4$. Thus the sequence $(\alpha_k)$ is log-concave for $r=2$.
		 For $r=3$, from \eqref{eq:ak-balanced}, $\alpha_0=1, \alpha_1=3t,$ and, for $k\ge 2$, $\alpha_k = \binom{t}{k}3^k + t\binom{3}{k}. $ In particular, $\alpha_2 = 9\binom{t}{2} + 3t= \frac{3}{2}t(3t-1), $
		$\alpha_3 = 27\binom{t}{3} + t,	\alpha_4 = 81\binom{t}{4},	\alpha_5 = 243\binom{t}{5},$ and so on.
		Hence,  $\alpha_k = 3^k\binom{t}{k}$ for all $k\ge 4$. For $k\ge 5$ the triple $(\alpha_{k-1},\alpha_k,\alpha_{k+1})$ lies entirely in the binomial
		regime and is log-concave by the preliminary fact. Thus, we only need to check $k=1,2,3,4$. For $k=1$, we have
		\[
		\alpha_1^2 - \alpha_0\alpha_2
		= (3t)^2 - \frac{3}{2}t(3t-1)
		= \frac{3t}{2}(3t+1)  \ge  0,
		\]
		which is true for all $t\ge 1$.
		
		For $k=2$,  we have
		\begin{align*}
			\alpha_2^2 - \alpha_1\alpha_3
			&= \Bigl(\frac{3}{2}t(3t-1)\Bigr)^2
			- (3t)\bigl(27\binom{t}{3}+t\bigr)= \frac{3t^2}{4}\bigl(9t^2+36t-37\bigr).
		\end{align*}
		For \( t \ge 2 \), we derive \( 9t^2 + 36t - 37 \ge 9t^2 + 36t - 45 = 9(t - 1)(t + 5) \ge 0 \), thus concluding that \( \alpha_2^2 \ge \alpha_1 \alpha_3 \).
		
		For $k=3$, we have
		\[
		\alpha_3^2 - \alpha_2\alpha_4
		= \frac{t^2}{16}\bigl(81t^4 -405t^3 +1197t^2 -1971t +1114\bigr)
		= \frac{t^2}{16}P_3(t).
		\]
		The quartic $P_3$ has positive leading coefficient, and 
		$P_3(2)=1600$, $P_3(3)=7198$, $P_3(4)=21184$, so $P_3(t)>0,$ for
		$t=2,3,4$. For $t\ge 5$, we split
		\[
		P_3(t) = 81t^4 -405t^3 + (1197t^2 -1971t +1114)
		\ge 81t^4 -405t^3.
		\]
		As the bracket is a quadratic with positive leading coefficient, and the
		minimum is attained at $t<2$. Hence, it is increasing on $[2,\infty)$, and already
		positive at $t=2$. Thus for $t\ge 5$,
		\[
		P_3(t) \ge 81t^4 -405t^3 = 81t^3(t-5)\ge 0.
		\]
		Combining these cases we have $P_3(t)>0$ for all $t\ge 2$, and it shows that 
		$\alpha_3^2\ge \alpha_2\alpha_4$.
		
		For $k=4$,  we have $\alpha_4 = 81\binom{t}{4},$ and $\alpha_5 = 243\binom{t}{5}.$  
		So, we have
		\[
		\alpha_4^2 - \alpha_3\alpha_5
		= \frac{81}{320} t^2(t-1)(t-2)(t-3) 
		\bigl(9t^3 -18t^2 -17t +50\bigr)
		=\frac{81}{320} t^2(t-1)(t-2)(t-3) Q_3(t).
		\]
		For $t\ge 3$ the factors $t^2(t-1)(t-2)(t-3)$ are nonnegative. For the cubic
		$Q_3$, we observe that, for $t\ge 3$,
		\[
		Q_3(t) \ge 9t^3 -18t^2 -17t
		= t(9t^2 -18t -17).
		\]
		The quadratic $9t^2 - 18t - 17$ possesses a vertex at $t = 1$, is growing for $t \ge 1$, and attains a value of $10$ at $t = 3$. Consequently, $Q_3(t)>0$ for all $t\ge 3$. For $t=2$, we verify that $\alpha_4=0$ and $\alpha_5=0$, so the inequality $\alpha_4^2 \ge \alpha_3 \alpha_5$ is trivially satisfied. Thus $(\alpha_k)$ is log-concave for $r=3$.
		
		For $r=4$, \eqref{eq:ak-balanced}, gives  $\alpha_0=1, \alpha_1=4t,$ and for $k\ge 2$,
		\[
		\alpha_k = \binom{t}{k}4^k + t\binom{4}{k}.
		\]
		In particular,
		$ \alpha_2 = 16\binom{t}{2} + 6t = 2t(4t-1), \alpha_3 = 64\binom{t}{3} + 4t,\alpha_4 = 256\binom{t}{4} + t,	\alpha_5 = 1024\binom{t}{5},\alpha_6 = 4096\binom{t}{6},$ and so on.
		Thus, we obtain  $\alpha_k = 4^k\binom{t}{k}$ for all $k\ge 5$. As before, for $k\ge 6$ the triple $(\alpha_{k-1},\alpha_k,\alpha_{k+1})$ lies entirely in the binomial regime, hence is log-concave. We therefore only need to check	$k=1,2,3,4,5$. For $k=1$, we have
		\[
		\alpha_1^2 - \alpha_0\alpha_2 = (4t)^2 - 2t(4t-1)=2t(4t+1)  \ge  0.
		\]
		For $k=2$, we have
		\begin{align*}
			\alpha_2^2 - \alpha_1\alpha_3
			&= \bigl(2t(4t-1)\bigr)^2 - (4t)\bigl(64\binom{t}{3}+4t\bigr) = \frac{4t^2}{3}\bigl(16t^2 + 72t - 73\bigr).
		\end{align*}
		For $t\ge 2$,
		\[
		16t^2 + 72t - 73  \ge  16t^2 + 72t - 80
		= 8(2t^2 + 9t - 10),
		\]
		and $2t^2 + 9t - 10 \ge 2\cdot 4 + 18 - 10 = 16>0$ for $t\ge 2$. Thus,
		$\alpha_2^2\ge \alpha_1\alpha_3$.
		
		For $k=3$, we have
		\[
		\alpha_3^2 - \alpha_2\alpha_4
		= \frac{2t^2}{9}\bigl(128t^4 - 672t^3 + 2240t^2 - 3972t + 2321\bigr)
		= \frac{2t^2}{9}P_4(t).
		\]
		The quartic polynomial \( P_4 \) possesses a positive leading coefficient, with \( P_4(2) = 9 \), \( P_4(3) = 2789 \), and \( P_4(4) = 12033 \). Thus, $P_4(t)>0$ for $t=2, 3, 4$. For $t \ge 5$, we partition
		\[
		P_4(t) = 128t^4 - 672t^3 + (2240t^2 - 3972t + 2321)
		\ge 128t^4 - 672t^3.
		\]
		The bracket represents a quadratic function in $t$ with a positive leading coefficient, indicating that the minimum occurs for $t<2$. Consequently, it is growing on the interval $[2,\infty)$ and is positive at $t=2$. Therefore, for \( t \geq 5 \),
		\[
		P_4(t) \ge 128t^4 - 672t^3 = 128t^3(t-5.25) > 0.
		\]
		Consequently, $P_4(t)>0$ for all integers $t\ge 2$, which indicates that $\alpha_3^2\ge \alpha_2\alpha_4$.
		
		For $k=4$, we have
		\[
		\alpha_4^2 - \alpha_3\alpha_5
		= \frac{t^2}{45}\Bigl(1024t^6 - 8192t^5 + 20992t^4 - 4160t^3
		- 70784t^2 + 116032t - 54867\Bigr)
		= \frac{t^2}{45}R_4(t).
		\]
		We express \( R_4(t) = 45 + W(t) \) and factor \( W(t) \) as follows: \[ W(t) = R_4(t) - 45 = 64(t-1)(t-2)(t-3)\bigl(16t^3 - 32t^2 - 40t + 143\bigr). \]
		For $t\ge 3$, the factor $(t-1)(t-2)(t-3)\ge 0$, and the cubic $16t^3 - 32t^2 - 40t + 143$ is increasing on $[3,\infty)$ (its derivative has positive leading coefficient and is already positive at $t=3$), and positive at $t=3$. Thus, 	$16t^3 - 32t^2 - 40t + 143>0$ for $t\ge 3$. For $t=2$, we have $W(2)=0$, so it gives
		$R_4(2)=45>0$. Hence, $R_4(t)>0$, for all integers $t\ge 2$, and thus we obtain	$\alpha_4^2\ge \alpha_3\alpha_5$.
		
		For $k=5$, $\alpha_5=1024\binom{t}{5}$ and $\alpha_6=4096\binom{t}{6}$, we have
		\begin{align*}
			\alpha_5^2 - \alpha_4\alpha_6&= \frac{256}{675} t^2(t-1)(t-2)(t-3)(t-4) 
		\bigl(32t^4 - 160t^3 + 160t^2 + 145t - 117\bigr)\\
		&= \frac{256}{675} t^2(t-1)(t-2)(t-3)(t-4) S_4(t).
		\end{align*}
		For the index $k=5$ to be relevant in the log-concavity condition we need
		$5\le d-1$, that is, $d\ge 6$. Since $r=4$, this forces $t\ge 6$, and thus
		all the factors $t^2(t-1)(t-2)(t-3)(t-4)$ are positive. The quartic $S_4$ has
		positive leading coefficient and one checks that
		$S_4(2)=975$, $S_4(3)=4608$, $S_4(4)=13425$, hence $S_4(t)>0$ for
		$t\ge 2$, and  in particular for $t\ge 6$. Therefore, we have $\alpha_5^2\ge \alpha_4\alpha_6$. Combining the four cases $r=1,2,3,4$, we conclude that for each $r\in\{1,2,3,4\}$ and every $t\ge 2$ the coefficient sequence $(\alpha_k)_{k=0}^d$	of $\psi(K_{r,\dots,r})$ is log-concave. Since all coefficients $\alpha_k$ with
		$0\le k\le d$ are nonzero, the sequence has no internal zeros, and therefore
		it is unimodal (see, \cite{stanley}).
	\end{proof}
	
	\begin{remark}
		The proof above uses only that the contribution \(t\binom{r}{k}\) is supported on a small initial segment of the index set when \(r\) is small, combined with the fact that the pure term \(\binom{t}{k}r^k\) is log-concave in \(k\). For larger \(r\), the support of the correction term overlaps significantly with that of \(\binom{t}{k}r^k\), and the simple argument used above is no longer sufficient. Indeed, counterexamples show that the conclusion fails in general.
	\end{remark}
	
	 In contrast to Theorem~\ref{thm:balanced-small-r}, the GPA of \(K_{r,\dots,r}\) is not always unimodal (or log-concave) when \(r\) is larger. We give explicit examples.
	
	Consider the balanced complete bipartite graph \(K_{8,8}\), which can be viewed as \(K_{r,\dots,r}\) with \(r=8\) and \(t=2\). From \eqref{eq:ak-balanced} derive that $\alpha_0=1, \alpha_1=16,$ and for \(k\ge 2\), we have $\alpha_k = \binom{2}{k}8^k +2\binom{8}{k}.$ The explicit polynomial is 
		\[
		\psi(K_{8,8})
		= 1 + 16x + 120x^2 + 112x^3 + 140x^4 + 112x^5 + 56x^6 + 16x^7 + 2x^8.
		\]
		The coefficient sequence $\{1, 16, 120, 112, 140, 112, 56, 16, 2\}$ is not unimodal, as it ascends from $1$ to $120$, descends to $112$, ascends again to $140$, and thereafter descends. Furthermore, \(112^{2} \ngeq 120 \cdot 140\). Therefore, neither unimodality nor log-concavity is applicable to \(\psi(K_{8,8})\).

	For \(r=5\) and \(t=2\), that is, for the balanced complete bipartite graph \(K_{5,5}\), a direct computation using \eqref{eq:ak-balanced} gives $\psi(K_{5,5}) = 1 + 10x + 45x^2 + 20x^3 + 10x^4 + 2x^5,$	with coefficient sequence $\{1, 10, 45, 20, 10, 2.\}$ The log-concavity inequality fails at index \(k=3\) as $\alpha_3^2 = 20^2 = 400 < 45\cdot 10 = 450 = \alpha_2 \alpha_4.$ Thus, \(\psi(K_{5,5})\) is not log-concave, although the sequence is still unimodal in this case.
	
	These examples show that, unlike in Theorem~\ref{thm:balanced-small-r}, the GPA of balanced complete multipartite graphs need not be unimodal or log-concave in general. Determining precisely for which pairs \((r,t)\) the polynomial \(\psi(K_{r,\dots,r})\) is unimodal or log-concave appears to be a delicate problem, and can be studied in future direction of research.

	\section{Log-concave and unimodal properties of general position polynomial}\label{sec 5}
	A broom \(B_{s,r}, s \geq 0, r \geq 0\), is a graph with vertices \(u_0, \ldots, u_s, v_1, \ldots, v_r\), and edges \(u_i u_{i+1}\) for \(i \in \{0, \ldots, s - 1\}\) and \(u_0 v_j\) for \(1\leq j \leq r\), see $B_{4,6}$ Figure\ref{fig:broom}. 	\begin{figure}[ht]
		\centering
		\begin{tikzpicture}[scale=1,
			every node/.style={circle,draw,inner sep=1pt,font=\small},
			>=stealth]
			
			\node (u0) at (0,0)  {$u_0$};
			\node (u1) at (1.5,0) {$u_1$};
			\node (u2) at (3,0)   {$u_2$};
			\node (u3) at (4.5,0) {$u_3$};
			\node (u4) at (6,0)   {$u_4$};
			
			\draw (u0) -- (u1) -- (u2) -- (u3) -- (u4);
			
			\node (v1) at (-1.1, 0.8) {$v_4$};
			\node (v2) at (-1.1, 0.0) {$v_4$};
			\node (v3) at (-1.1,-0.8) {$v_6$};
			\node (v4) at ( 0.9,  1.1) {$v_1$};
			\node (v5) at ( 0.0,  1.8) {$v_2$};
			\node (v6) at (-0.9, 1.4) {$v_3$};
			
			\draw (u0) -- (v1);
			\draw (u0) -- (v2);
			\draw (u0) -- (v3);
			\draw (u0) -- (v4);
			\draw (u0) -- (v5);
			\draw (u0) -- (v6);
			
		\end{tikzpicture}
		\caption{The broom graph $B_{4,6}$.}
		\label{fig:broom}
	\end{figure}
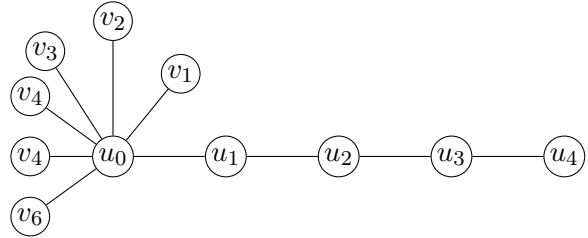
	The GPA of broom is (see \cite{0})
	\begin{equation}\label{broom}
		  \psi(B_{s,r}) = \sum_{k\geq 0} \beta_k x^k = 1 + (s + r + 1)x + \binom{s + r + 1}{2}x^2 + \sum_{k\geq 3} \left[s\binom{r}{k-1} + \binom{r}{k}\right] x^k. 
	\end{equation}

The following result discuss the unimodal and log-concave property of $ \psi(B_{s,r}).$
\begin{theorem}\label{prop:broom-unimodality-logconcavity-correct}
	Let $B_{s,r}$ be the broom graph and let $\psi(B_{s,r})=\sum_{k\ge 0} \beta_k x^k$ be its GPA. With $n=s+r+1$  the following assertions hold.
	\begin{enumerate}
		\item For all integers $r\ge 0$ and $s\ge 0$, the sequence $(\beta_k)_{k\ge 0}$ is
		log–concave at indices $k=1$ and $k=2$.
		\item For all integers $r\ge 0$ and $s\ge 0$, the tail subsequence
		$(\beta_k)_{k\ge 4}$ is log–concave. In particular, all possible failures of
		log–concavity can occur only at index $k=3$.
		
		\item For $r\le 2$ and all $s\ge 0$, the full sequence $(\beta_k)$ is
		log–concave and hence unimodal.
		
		\item For every fixed $r\ge 3$ there exists $S(r)$ such that for all
		$s\ge S(r)$ the sequence $(\beta_k)$ is not log–concave (log–concavity fails at
		$k=3$).
		
		\item For $r\le 5$ and all $s\ge 0$, the sequence $(\beta_k)$ is unimodal.
		
		\item For every $r\ge 6$ there exists $S'(r)$ such that for all
		$s\ge S'(r)$ the sequence $(\beta_k)$ is not unimodal. In particular, the polynomial  $\psi(B_{17,6}) $ is neither log–concave  nor unimodal.
	\end{enumerate}
\end{theorem}

\begin{proof}
	Throughout the proof,  we use \eqref{broom}, so with $n=s+r+1$, we have
	\[
	\beta_0=1,\quad \beta_1=n,\quad \beta_2=\binom{n}{2},\quad
	\beta_k=\binom{r}{k}+s\binom{r}{k-1}\ (k\ge 3).
	\]
	For $k=1$, we have
	\[
	\beta_1^2 - \beta_0\beta_2
	= n^2 - \binom{n}{2}
	= n^2 - \frac{n(n-1)}{2}
	= \frac{n(n+1)}{2}\ge 0, 
	\]
	which holds for all integers $n\ge 0$. Thus log–concavity holds at $k=1$.
	For, $k=2$, using $\beta_3=\binom{r}{3}+s\binom{r}{2}$, a straightforward  calculation gives
	\begin{align*}
		\beta_2^2 - \beta_1\beta_3
		&= \binom{n}{2}^2 - n\Bigl(\binom{r}{3}+s\binom{r}{2}\Bigr) = \frac{(r+s+1) P_2(r,s)}{12},
	\end{align*}
	where
	\[
	P_2(r,s)
	= r^3 + 3r^2 s + 9r^2 + 9r s^2 + 12rs - 4r + 3s^3 + 3s^2.
	\]
	All monomials in $P_2(r,s)$ have nonnegative coefficients except $-4r$. At	$s=0$, we have $P_2(r,0) = r(r^2+9r-4)>0$, for every integer $r\ge 1$. For $s>0$ the extra terms
	$3r^2 s + 9rs^2 + 12rs + 3s^3 + 3s^2$ are positive, so $P_2(r,s)>0$ for all
	$r\ge 1$, $s\ge 0$. Hence, $\beta_2^2\ge \beta_1\beta_3$ holds.\\
	(2) For $k\ge 4$, consider the auxiliary polynomial
	\[
	q(x)=(1+sx)(1+x)^r = \sum_{k\ge 0} \gamma_k x^k.
	\]
	It is clear that
	\[
	\gamma_0=1,\quad \gamma_1=s+r,\quad
	\gamma_k=\binom{r}{k}+s\binom{r}{k-1}\quad(k\ge 1).
	\]
	Thus $\gamma_k=\beta_k$, for all $k\ge 3$, and in particular $\gamma_{k-1}=\beta_{k-1},
	\gamma_k=\beta_k,\gamma_{k+1}=\beta_{k+1}$, for every $k\ge 4$. The roots of $(1+x)^r$ are all real ($x=-1$ with multiplicity $r$), and the	root of $(1+sx)$ is real ($x=-1/s$ for $s>0$). Hence $q(x)$ has only real
	roots, and it is well known by Newton's Inequality that its coefficient sequence $(\gamma_k)$ is log–concave \cite{stanley, hardy}. Therefore for every $k\ge 4$,
	\[
	\beta_k^2 - \beta_{k-1}\beta_{k+1}
	= \gamma_k^2 - \gamma_{k-1}\gamma_{k+1}\ge 0.
	\]
	This proves (2), as any violation of log–concavity can occur only at index $k=3$.
	
	For  $r\le 2$, we already know that log–concavity holds at $k=1,2$ and at all $k\ge 4$. Thus, we only need to check the remaining index $k=3$ when it exists.	If $r=0$, the graph is a path and $\deg\psi(B_{s,0})\le 2$, so log–concavity is trivial. If $r=1$, then $\deg\psi(B_{s,1})\le 2$ again path and the claim is
	trivial. If $r=2$, we have
	\[
	\beta_3 = \binom{2}{3} + s\binom{2}{2} = s,\quad
	\beta_4 = \binom{2}{4} + s\binom{2}{3} = 0,
	\]
	so, we obtain $\beta_3^2 - \beta_2\beta_4 = s^2\ge 0.$ Hence, $(\beta_k)$ is log–concave for all $s\ge 0,$ when $r\le 2$, proving (3). In	particular it is also unimodal.
	
	At $k=3$ for $r\ge 3$ and large $s$, we consider $\beta_3^2 - \beta_2\beta_4,$ where $\beta_3=\binom{r}{3}+s\binom{r}{2}, $ and $
	\beta_4=\binom{r}{4}+s\binom{r}{3}.$ Now, we have $\beta_3^2 - \beta_2\beta_4
	= \frac{r(r-1)Q(r,s)}{144},$ where
	\begin{align*}
		Q(r,s)
		=& r^4 + r^3(6s-8) + r^2(9s^2 - 33s + 29) + r(-12s^3 + 15s^2 + 51s - 34)\\
		&+ 24s^3 + 6s^2 - 18s.
	\end{align*}
	For fixed $r$, $Q(r,s)$ is a cubic polynomial in $s$, whose leading term is $(24-12r)s^3.$ If $r\ge 3$, then $24-12r<0$, so $Q(r,s)<0$, for all sufficiently large $s$. Hence, $\beta_3^2 - \beta_2\beta_4<0$ for all large $s$, log–concavity fails at $k=3$. This proves (4).
	
	For unimodality with $r\le 5$, we briefly analyze the shape of $(\beta_k)$, for each $r\le 5$. First note that for all $s,r\ge 0$, $\beta_0<\beta_1<\beta_2$ as soon as $n=s+r+1\ge 3$, and the remaining small cases are trivial. For $k\ge 3$, the coefficients are given by \eqref{broom}. For small $r$, we can write them explicitly. If $r=3$, then $\beta_3=1+3s,\beta_4=s,$ and $\beta_k=0,$ for $k\ge 5.$  As $\beta_2>\beta_3$, we see that 	$\beta_2-\beta_3=\tfrac12(s^2+s+10)>0$, and clearly $\beta_3>\beta_4$. Thus, we obtain  $\beta_0<\beta_1<\beta_2>\beta_3>\beta_4.$ So, $(\beta_k)$ is unimodal for all $s\ge 0$. If $r=4$, then $\beta_3=4+6s,  \beta_4=1+4s, \beta_5=s, $ and $\beta_k=0$ for $k\ge 6,$  
	and we have
	\[
	\beta_2-\beta_3=\tfrac12(s^2-3s+12)>0,\quad
	\beta_3-\beta_4=3+2s>0,\quad
	\beta_4-\beta_5=1+3s>0.
	\]
	Hence, we obtain  $\beta_0<\beta_1<\beta_2>\beta_3>\beta_4>\beta_5,$ and $(\beta_k)$ is unimodal for all $s\ge 0$. If $r=5$, then $\beta_3=10+10s,  \beta_4=5+5s,  \beta_5=1+s,  \beta_6=s,$ and $\beta_k=0$ for $k\ge 7.$ So, we obtain $\beta_3>\beta_4>\beta_5>\beta_6$, for all $s\ge 0$, while $\beta_2-\beta_3=\tfrac12(s^2-9s+10)$ may be positive or negative. Thus we have two possibilities: if $\beta_2\ge \beta_3$, then $\beta_0<\beta_1\le \beta_2>\beta_3>\beta_4>\beta_5>\beta_6$ and the peak is at $k=2$, and   if $\beta_2<\beta_3$, then $\beta_0<\beta_1<\beta_2<\beta_3>\beta_4>\beta_5>\beta_6$ and the peak is at $k=3$. In either case, there is a unique index $m\in\{2,3\}$ such that $\beta_0\le\cdots\le \beta_m\ge\cdots\ge \beta_{\deg\psi}$, so $(\beta_k)$ is unimodal. The cases $r=0,1,2$ are trivial and are covered by (3). This proves (5).
	
	For $r\ge 6$ and large $s$, we fix $r\ge 6$ and let $s\to\infty$. From \eqref{broom}, we have
	\[
	\beta_2=\binom{n}{2}=\frac{(s+r+1)(s+r)}{2},\quad
	\beta_3=\binom{r}{3}+s\binom{r}{2},\quad
	\beta_4=\binom{r}{4}+s\binom{r}{3}.
	\]
	With asymptotics values, we assume that 
	\[
	\beta_2 \sim \frac{s^2}{2},\quad
	\beta_3 \sim s\binom{r}{2},\quad
	\beta_4 \sim s\binom{r}{3}\qquad (s\to\infty).
	\]
	Hence, we obtain	$\frac{\beta_3}{\beta_2} \sim \tfrac{2\binom{r}{2}}{s}\to 0.$
	So, for all sufficiently large $s$, we have $\beta_2>\beta_3$. On the other hand, for every $r\ge 6$, we have
	\[
	\frac{\beta_4}{\beta_3} \sim \frac{\binom{r}{3}}{\binom{r}{2}} = \frac{r-2}{3} > 1.
	\]
	So, for all sufficiently large $s$, we have $\beta_4>\beta_3$. Therefore, for all large $s$, we obtain
	\[
	\beta_0<\beta_1<\beta_2>\beta_3<\beta_4,
	\]
	thereby the coefficient sequence is not unimodal. This gives us the existence of $S'(r)$ as claimed in (6).  For every $r\ge 6,$ there exists $S'(r)$ such that for all $s\ge S'(r)$ the sequence $(\beta_k)$ is not unimodal. In particular, for
	$B_{17,6}$ the polynomial 
	\[
	\psi(B_{17,6}) = 1 + 24x + 276x^2 + 275x^3 + 355x^4 + 261x^5 + 103x^6 + 17x^7,
	\]
	whose coefficient sequence is neither log–concave (since $275^2 < 276\cdot 355$) nor unimodal (it satisfies
	$\beta_1<\beta_2>\beta_3<\beta_4>\beta_5>\cdots$).
\end{proof}

		Let $G_n$ be the comb on $2n$ vertices obtained from the path $P_n$ by
		attaching a pendent vertex to each vertex of $P_n$, for $n\ge 1$. The unimodality of $G_{n}$ is proved in \cite[Theorem 4.1]{0}. We will show its GPA is log-concave.		
		\begin{theorem}\label{thm:comb-unimodal-correct}
			The GPA of $G_n$  is $\psi(G_n)=\sum_{k\ge 0} \alpha_k x^k,$ where $\alpha_0 = 1,  \alpha_1 = 2n,  \alpha_2 = \binom{2n}{2},$ 	and,  $\alpha_k = 4\binom{n}{k},$ for $k\ge 3$.  The coefficient sequence $(\alpha_0,\dots,\alpha_n)$ of $\psi(G_n)$ is log-concave for all  $n\ge 1$, and hence $\psi(G_n)$ is unimodal.
		\end{theorem}
		
		\begin{proof}
			We  only need to work with general position sets of cardinality  $k\ge 3$. We label the vertices of $P_n$ by $p_1,\dots,p_n$ in order and the pendent
			neighbors by $q_1,\dots,q_n$, where $q_i$ is adjacent only to $p_i$. Let $S\subseteq V(G_n)$ be a general position set with $|S|\ge 3$.
			We first show that $S$ contains at most two path vertices $p_i$. If $S$ contained three distinct path vertices $p_i,p_j,p_\ell$ with $i<j<\ell$, then the unique shortest path between $p_i$ and $p_\ell$ in $G_n$ is the subpath $p_i\sim p_{i+1}\sim \dots\sim p_\ell$, on which $p_j$ lies. Thus, $d(p_i,p_\ell)=d(p_i,p_j)+d(p_j,p_\ell).$ So, $p_j$ lies on a geodesic between $p_i$ and $p_\ell$, contradicting that $S$ is in general position. Hence, at most two $p_i$ can belong to $S$. We now refine the above fact, and claim that: $S$ be a general position set with $|S|\ge 3$ and let
				$I=\{i\in[n]: p_i\in S\text{ or }q_i\in S\}$ be the set of \emph{indices}
				used by $S$. Then $|I|\ge 3$ and, if $i\in I$ is not the minimum or maximum
				of $I$, we must have $q_i\in S$ and $p_i\notin S$. \\
				Proof of Claim: Since $|S|\ge 3$ and each index contributes at most two vertices, we clearly
				have $|I|\ge 2$. If $|I|=2$ and one of these indices contributes two vertices,	then the third vertex must come from one of these indices, but then some	vertex lies on a geodesic between the other two (this can be checked directly	on a path with one pendent neighbor at each endpoint). Thus for $|S|\ge 3$
				we actually have $|I|\ge 3$. Let $i\in I$ be neither the smallest nor the largest index of $I$. Then there
				exist $u<v$ in $I$ with $u<i<v$. Suppose $p_i\in S$. We consider two cases. (1) Either $p_u\in S$ or $p_v\in S$. Then by the same argument as in the  proof of claim, $p_i$ lies on the unique shortest	path between $p_u$ and $p_v$. So, $\{p_u,p_i,p_v\}$ is not in general position, a	contradiction. (2) If $p_u,p_v\notin S$, then $q_u,q_v\in S$. The unique shortest	path between $q_u$ and $q_v$ is $q_u\sim p_u\sim p_{u+1}\sim \dots\sim p_v\sim q_v,$ and again $p_i$ lies internally on this path. So, we obtain $d(q_u,q_v)=d(q_u,p_i)+d(p_i,q_v),$ which contradicts that $S$ is in general position. Thus, if $i$ is not the minimum or maximum index in $I$, it is impossible that $p_i\in S$, and hence necessarily $q_i\in S$. The above claim  shows that for any general position set 	$S$ of size $k\ge 3$ the set of indices $I\subseteq\{1,2,\dots,n\}$ with $|I|=k$ can be	chosen arbitrarily, and then: (1) for each \emph{internal} index in $I$ we must take the leaf $q_i$; (2) for the two \emph{extreme} indices (the minimum and maximum of $I$), we may choose either the path vertex $p_i$ or the leaf $q_i$. Thus, for any fixed $k$-subset $I=\{i_1<\dots<i_k\}\subseteq [n]$ with	$k\ge 3$, there are exactly $4$ admissible choices for $S$:
			\[
			(p_{i_1}\ \text{or}\ q_{i_1}),\quad
			q_{i_2},\dots,q_{i_{k-1}},\quad
			(p_{i_k}\ \text{or}\ q_{i_k}),
			\]
			and no other choices yield a general position set of size $k$. Since there
			are $\binom{n}{k}$ ways to choose $I$, we obtain, $\alpha_k = 4\binom{n}{k},$ for $k\ge 3$. 
			
			 We now show that $(\alpha_k)_{k=0}^n$ is log-concave.  For $k\ge 4$, we have $\alpha_{k-1}=4\binom{n}{k-1}$, $\alpha_k=4\binom{n}{k}$,
			$\alpha_{k+1}=4\binom{n}{k+1}$, so
			\[
			\alpha_k^2  =  16\binom{n}{k}^2
			 \ge  16\binom{n}{k-1}\binom{n}{k+1}
			 =  \alpha_{k-1}\alpha_{k+1},
			\]
			since the binomial coefficients form a log-concave sequence in $k$. Consequently, log-concavity is trivially satisfied for \(4 \leq k \leq n-1\), provided those indices fall inside the binomial range. Therefore, it is sufficient to validate the inequalities for $k=1,2,3$. For \( k=1 \), we obtain
			\[
			\alpha_1^2 - \alpha_0\alpha_2 = (2n)^2 - \binom{2n}{2} = 4n^2 - \frac{2n(2n-1)}{2} 
			= n(2n+1)\ge 0,
			\]
			which holds for all $n\ge 1$. If $k=2$, then $\alpha_2=\binom{2n}{2}$ and
			$\alpha_3=4\binom{n}{3}=\frac{2}{3}n(n-1)(n-2)$, so we obtain
			\begin{align*}
				\alpha_2^2 - \alpha_1\alpha_3
				&= \binom{2n}{2}^2 - (2n)\cdot 4\binom{n}{3} = \frac{n^2}{3}\bigl(8n^2 - 5\bigr).
			\end{align*}
			Because $8n^2-5>0$ for every $n\ge 1$, it follows that $\alpha_2^2\ge \alpha_1\alpha_3$. For \( k=3 \), \( \alpha_2 = \binom{2n}{2} \), \( \alpha_3 = 4\binom{n}{3} \), and \( \alpha_4 = 4\binom{n}{4} \). Consequently, we possess
			\[
			\alpha_3^2 - \alpha_2\alpha_4
			= 16\binom{n}{3}^2 - \binom{2n}{2}\cdot 4\binom{n}{4}
			= \frac{n^2(n-1)(n-2)}{18}\bigl(2n^2 - 3n + 7\bigr).
			\]
			The quadratic \(2n^2 - 3n + 7\) has a discriminant \(\Delta = (-3)^2 - 4 \cdot 2 \cdot 7 = 9 - 56 < 0\) and a positive leading coefficient; therefore, it is positive for any real \(n\). Given that $n^2(n-1)(n-2)\ge 0$ for all integers $n\ge 2$, it follows that $\alpha_3^2\ge \alpha_2\alpha_4$.  Thus, the coefficient sequence $(\alpha_0,\dots,\alpha_n)$ is log-concave. All coefficients
			are strictly positive, hence there are no internal zeros, and therefore the sequence is unimodal.
		\end{proof}

	The vertex set of Kneser graph \( K(n, 2) \) contains all 2-subsets of an \( n \)-set, two vertices being adjacent if the corresponding sets are disjoint. A Kneser graph graph $K_{5,2}$ is shown in Figure \ref{fig:kneser-5-2}.	
	\begin{figure}[H]
		\centering
		\begin{tikzpicture}[scale=2,
			every node/.style={circle,draw,inner sep=1.5pt,font=\scriptsize},
			>=stealth]
			
			\foreach \i/\lbl in {0/{01},1/{12},2/{23},3/{34},4/{40}}{
				\node (u\i) at (72*\i:1.2) {$\lbl$};
			}
			
			\foreach \i/\lbl in {0/{02},1/{24},2/{41},3/{13},4/{30}}{
				\node (v\i) at (72*\i+36:0.6) {$\lbl$};
			}
			
			\foreach \i in {0,1,2,3,4}{
				\pgfmathtruncatemacro{\j}{mod(\i+1,5)}
				\draw (u\i) -- (u\j);
			}
			
			\foreach \i in {0,1,2,3,4}{
				\pgfmathtruncatemacro{\j}{mod(\i+1,5)}
				\draw (v\i) -- (v\j);
			}
			
			\foreach \i in {0,1,2,3,4}{
				\draw (u\i) -- (v\i);
			}
			
		\end{tikzpicture}
		\caption{The Kneser graph $K(5,2)$, isomorphic to the Petersen graph. Each vertex is a $2$-subset of $\{0,1,2,3,4\}$, and two vertices are adjacent if the corresponding subsets are disjoint.}
		\label{fig:kneser-5-2}
	\end{figure}
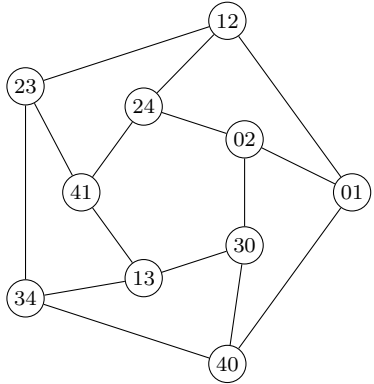
	
The unimodal of $\psi(K(n,2))$ is proved in \cite[Theorem 4.2]{0}. 	We now observe that, despite unimodality, log-concavity fails in general.  The GPA is 
\begin{align*}
	\psi(K(n, 2))
	&= 1 + \binom{n}{2} x  + \left( \binom{n}{3} + 12 \binom{n}{4} \right)x^3
	+ 15 \binom{n}{4} x^4
	+ 6 \binom{n}{4} x^5
	+ \binom{n}{4} x^6 \\
	&\quad + \sum_{j=2}^{n-1} \left( \binom{n}{2j} \frac{(2j)!}{2^j j!}
	+ n \binom{n-1}{j} \right)x^j
\end{align*}
	
		From the above expression for $K(10,2)$, we obtain
		\[
		\psi(K(10,2))
		= 1 + 45x + 990x^2 + 6630x^3 + 9135x^4
		+ 3465x^5 + 1050x^6 + 360x^7 + 90x^8 + 10x^9.
		\]
		At $k=6$, the condition  $\alpha_6^2 \ge \alpha_5 \alpha_7$  is not satisfied, since 
		\[
		\alpha_6^2 = 1050^2 = 1102500 \ngeq 	\alpha_5 \alpha_7 = 3465\cdot 360 = 1247400.
		\]
		
	Review the family of graphs $T^*(r,a)$, defined as a full $a$-partite graph, with each partition containing $r$ vertices. Label the vertices in part $i$ by $i_1,\dots,i_r$. For each $1\leq j\in r$, delete the edges of the clique induced by	$\{1_j,2_j,\dots,\alpha_j\}$,  the order of $T^*(r,a)$ is $n=ra$. Now, we investigate unimodality and log-concavity of the GPAs of the family of graphs $T^*(r,a)$. For $a\in \{1,2\}$, it is known that $\psi(T^*(r,a))$ is unimodal \cite[Proposition 4.3]{0}.
	
	The GPA of $T^*(r,a)$ is (see \cite{0})
	\begin{align*}
		\psi(T^*(r,a))
		&= 1 + nx + \binom{n}{2}x^2 + 2a(a-1)\binom{r}{2}x^3 + \binom{a}{2}\binom{r}{2}x^4 + \sum_{i\ge 3} \Bigl[a\binom{r}{i} + r^i\binom{a}{i}\Bigr] x^i.
	\end{align*}
	
	We will show that for $a\ge 3$ this family does not, in general, have unimodal
	or log-concave GPAs. This concerns Problem 5.2 \cite{0}.	However, we note that $\psi(T^*(r,a))$ is not accurate and needs refinement. Similarly, \cite[Proposition 4.3]{0} must be revised, since $\psi(T^*(r,a))$ is not valid. Our next result gives the refined GPA of $T^*(r,a).$

		\begin{theorem}\label{thm:correct-gpa-Tstar}
			Let $T^*(r,a)$ be the graph with vertex set
			$
			V\bigl(T^*(r,a)\bigr)=\{ i_j:\ i\in [a],\ j\in [r] \},
			$
			where two distinct vertices $i_j$ and $k_\ell$ are adjacent if and only if
			$
			i\neq k $ and $ j\neq \ell.
			$
		If $\psi\bigl(T^*(r,a)\bigr)=\sum_{k\ge 0}\alpha_k x^k,$ then $\alpha_0=1$, $\alpha_1=ar$, and
			$
			\alpha_2=\binom{ar}{2}.
			$
			Moreover, the  GPA of $T^*(r,a)$ is given by the following three cases.
			\begin{enumerate}[label=\textup{(\roman*)}]
				\item If $\min\{r,a\}=1$, then $T^*(r,a)\cong \overline{K}_{ar}$ and hence
			$\psi\bigl(T^*(r,a)\bigr)=(1+x)^{ar}. $
			\item If $\min\{r,a\}=2$, and $m=\max\{r,a\}$, then $T^*(r,a)\cong K_{m,m}-M$, where $M$ is a perfect matching, and
			$$
			\psi\bigl(T^*(r,a)\bigr)
			=
			1+2mx+\binom{2m}{2}x^2+2\sum_{k=3}^{m}\binom{m}{k}x^k.
			$$
			\item If $r\ge 3$ and $a\ge 3$, then for every $k\ge 3$,
			$$
			\alpha_k
			=
			a\binom{r}{k}
			+
			r\binom{a}{k}
			+
			k!\binom{a}{k}\binom{r}{k}
			+
			\binom{a}{2}\binom{r}{2}\binom{4}{k},
			$$
			and
			{\footnotesize$$
			\psi\bigl(T^*(r,a)\bigr)
			=
			1+arx+\binom{ar}{2}x^2
			+
			\sum_{k=3}^{\max\{a,r,4\}}
			\left(
			a\binom{r}{k}
			+
			r\binom{a}{k}
			+
			k!\binom{a}{k}\binom{r}{k}
			+
			\binom{a}{2}\binom{r}{2}\binom{4}{k}
			\right)x^k.
			$$}
			\end{enumerate}
		\end{theorem}
		
		\begin{proof}
			The values $\alpha_0=1$, $\alpha_1=ar$, and $\alpha_2=\binom{ar}{2}$ hold for every graph. 	If $\min\{r,a\}=1$, then every two distinct vertices agree either in the unique part-index or in the unique label-index, so no two vertices are adjacent. Thus, $T^*(r,a)$ is edgeless and $\psi(T^*(r,a))=(1+x)^{ar}$. Now, assume that $\min\{r,a\}=2$, and let $m=\max\{r,a\}$. Then $T^*(r,a)$ is isomorphic to $K_{m,m}$ with a perfect matching removed. Let $X$ and $Y$ be the two bipartition classes. Any subset of $X$ is a general position set, and any subset of $Y$ is a general position set. Conversely, if a set $S$ contains two vertices $u,v\in X$ and a vertex $w\in Y$, then $u$ and $v$ have distance $2$, and every vertex of $Y$ that is adjacent to both lies on a shortest $u$--$v$ path. As $|S| \ge 3$, this precludes $S$ from intersecting both sides. Hence, every general position set of size at least $3$ is contained entirely in one side, and therefore
			$ \alpha_k=2\binom{m}{k}$ for $k\ge 3.$ Assume that $r\ge 3$ and $a\ge 3$. For $i\in [a]$, let
			$R_i=\{ i_1,\dots,i_r \}$ be the $i$th part, and for $j\in [r]$, let $C_j=\{ 1_j,\dots,a_j \}$	be the $j$th label-class. For two vertices $i_j$ and $k_\ell$ we have: (i) $d(i_j,k_\ell)=1$ if $i\neq k$ and $j\neq \ell$; and (ii) $d(i_j,k_\ell)=2$ if $i=k$ and $j\neq \ell$, or if $i\neq k$ and $j=\ell$. Let $S$ be a general position set with $|S|\ge 3$.	If $S$ contains two vertices from the same row $R_i$ and also a vertex outside $R_i$, then that outside vertex is adjacent to both row-vertices unless it lies in one of the two corresponding columns. In such cases, we see that any further vertex outside the resulting $2\times 2$ rectangle creates a geodesic. Hence, either $S\subseteq R_i$, or all vertices of $S$ lie in a $2\times 2$ rectangle determined by two rows and two columns. By symmetry, either $S\subseteq C_j$, or all vertices of $S$ lie in a $2\times 2$ rectangle.
			 If $S$ contains no two vertices from the same row and no two from the same column, then every two vertices of $S$ are adjacent, so $S$ is a clique. Such a clique of size $k$ is obtained by choosing $k$ rows, $k$ columns, and a bijection between them, giving $k!\binom{a}{k}\binom{r}{k}$	possibilities. If $S$ lies in a $2\times 2$ rectangle, then every $k$-subset of the four rectangle-vertices is in general position for $k=3,4$. The number of such rectangles is
			$\binom{a}{2}\binom{r}{2},$	and each contributes $\binom{4}{k}$ sets of size $k$. Therefore, for every $k\ge 3$, the general position sets are exactly of the following four mutually exclusive types: (a)  subsets of a single row $R_i$, counted by $a\binom{r}{k}$, (b)  subsets of a single column $C_j$, counted by $r\binom{a}{k}$, (c)   cliques of size $k$, counted by $k!\binom{a}{k}\binom{r}{k}$, (c) and   $k$-subsets of $2\times 2$ rectangles, counted by $\binom{a}{2}\binom{r}{2}\binom{4}{k}$. Summing these contributions, we obtain
			$$
			\alpha_k = a\binom{r}{k} + 	r\binom{a}{k} + 	k!\binom{a}{k}\binom{r}{k} + \binom{a}{2}\binom{r}{2}\binom{4}{k}.
			$$
			In particular, we have
			$$
			\operatorname{gp}\bigl(T^*(r,a)\bigr)=
			\begin{cases}
				ar, & \min\{r,a\}=1,\\
				\max\{r,a\}, & \min\{r,a\}=2,\\
				\max\{r,a,4\}, & r\ge 3,\ a\ge 3.
			\end{cases}
			$$
		\end{proof}
		
		\begin{remark}\label{rem:old-vs-new-Tstar}
			The formula previously stated in \cite{0}
			$$
			\psi_{\mathrm{old}}\bigl(T^*(r,a)\bigr)
			=
			1+arx+\binom{ar}{2}x^2
			+2a(a-1)\binom{r}{2}x^3
			+\binom{a}{2}\binom{r}{2}x^4
			+\sum_{k\ge 3}\left(a\binom{r}{k}+r^k\binom{a}{k}\right)x^k.
			$$
			For $r,a\ge 3$, this is not correct. The error is already visible in the coefficient of $x^3$. The old formula gives
			$$
			[x^3]\psi_{\mathrm{old}}\bigl(T^*(r,a)\bigr)
			=
			2a(a-1)\binom{r}{2}
			+
			a\binom{r}{3}
			+
			r^3\binom{a}{3},
			$$
			whereas the correct coefficient is
			$$
			[x^3]\psi\bigl(T^*(r,a)\bigr)
			=
			4\binom{a}{2}\binom{r}{2}
			+
			a\binom{r}{3}
			+
			r\binom{a}{3}
			+
			6\binom{a}{3}\binom{r}{3}.
			$$
			The old term $r^3\binom{a}{3}$ overcounts the clique-type contribution. The correct clique contribution is
			$
			r\binom{a}{3}+6\binom{a}{3}\binom{r}{3},
			$	not $r^3\binom{a}{3}$.
		\end{remark}
		
		\begin{example}
			As $T^*(3,2)\cong C_6$, we find its GPA by using definition. Let $ V(C_6)=\{v_1,v_2,v_3,v_4,v_5,v_6\}, $	where the cycle is $	v_1v_2v_3v_4v_5v_6v_1. $ By definition, $
			\psi(C_6)=\sum_{k\ge 0}\alpha_k x^k,	$	where $\alpha_k$ is the number of general position sets of size $k$.
			 Trivially, $\alpha_0=1, \alpha_1=6, \alpha_2=\binom{6}{2}=15,$ since every set of size at most $2$ is in general position.	Now, consider $3$-subsets. 	A $3$-set $S=\{x,y,z\}$ is in general position if no one of its vertices lies on a shortest path between the other two. In $C_6$, the only $3$-subsets with this property are the two alternating sets
			$	\{v_1,v_3,v_5\}$ and $ \{v_2,v_4,v_6\}.	$
			Indeed, in each of these sets every pair of vertices is at distance $2$, so no vertex lies on a geodesic between the other two.
			 On the other hand, any other $3$-subset contains two vertices at distance $1$ or $2$ and a third vertex that lies on a shortest path between them. For instance, the set $\{v_1,v_2,v_4\}$ is not in general position, as $d(v_1,v_4)=3=d(v_1,v_2)+d(v_2,v_4)=1+2.$	Similarly, $\{v_1,v_3,v_6\}$ is not in general position, since
			$d(v_3,v_6)=3=d(v_3,v_1)+d(v_1,v_6)=2+1.$ Hence, $\alpha_3=2.$ Finally, no set of size at least $4$ can be in general position in $C_6$, as any $4$-subset contains three vertices among which one lies on a shortest path between the other two. Thus,  $	\alpha_k=0 $ for all $k\ge 4.$	So, we have	$	\psi\bigl(T^*(3,2)\bigr)=\psi(C_6)=1+6x+15x^2+2x^3,$ which validates Theorem \ref{thm:correct-gpa-Tstar}.	By contrast, the old formula gives
			$$
			\psi_{\mathrm{old}}\bigl(T^*(3,2)\bigr)
			=
			1+6x+15x^2+14x^3+3x^4.
			$$
			Which is already false in this smallest nontrivial case.
		 Similarly, for  $T^*(3,3)$, Theorem~\ref{thm:correct-gpa-Tstar} implies that 
		 $$
			\alpha_3
			=
			3\binom{3}{3}
			+
			3\binom{3}{3}
			+
			3!\binom{3}{3}\binom{3}{3}
			+
			\binom{3}{2}\binom{3}{2}\binom{4}{3}
			=
			3+3+6+36=48,
			$$
			and $\alpha_4=\binom{3}{2}\binom{3}{2}=9.
			$
			Therefore, $\psi\bigl(T^*(3,3)\bigr)=1+9x+36x^2+48x^3+9x^4$ while  $	\psi_{\mathrm{old}}\bigl(T^*(3,3)\bigr)=1+9x+36x^2+66x^3+9x^4$ is not correct as  coefficient of $x^3$ is overcounted by $18$.
		 
		 Consider one more case $T^*(11,3)$,  we have
			$$
			\alpha_3
			=
			3\binom{11}{3}
			+
			11\binom{3}{3}
			+
			3!\binom{3}{3}\binom{11}{3}
			+
			\binom{3}{2}\binom{11}{2}\binom{4}{3}.
			$$
			Thus $	\alpha_3	=	3\cdot 165+11+6\cdot 165+3\cdot 55\cdot 4=2156, $  $\alpha_4=3\binom{11}{4}	+		\binom{3}{2}\binom{11}{2}  =1155,$ 
			and for $5\le k\le 11$, $ 	\alpha_k=3\binom{11}{k}. $ Hence, by Theorem \ref{thm:correct-gpa-Tstar}, we have
			 \begin{align*}
			 	\psi\bigl(T^*(11,3)\bigr) =& 1+33x+528x^2+2156x^3+1155x^4+1386x^5+1386x^6+990x^7+495x^8\\
			 	&+165x^9+33x^{10}+3x^{11}.
			 \end{align*} 
			The previous formula form \cite{0} gives, 
			 \begin{align*}
			 		\psi_{\mathrm{old}}\bigl(T^*(11,3)\bigr)= & 1+33x+528x^2+2486x^3+1155x^4+1386x^5+1386x^6+990x^7+495x^8\\
			 		&+165x^9+33x^{10}+3x^{11}.
			 \end{align*} 
			So again the discrepancy occurs in the coefficient of $x^3$, where the old formula exceeds the correct value by
			$ 2486-2156=330. $
		\end{example}

			In the next result, we discuss the unimodality and log-concave property of $\psi\bigl(T^*(r,a)\bigr).$
		\begin{theorem}\label{thm:Tstar-shape}
			Let $\psi\bigl(T^*(r,a)\bigr)=\sum_{k\ge 0}\alpha_kx^k$	be the  GPA of $T^*(r,a)$ as in Theorem~\ref{thm:correct-gpa-Tstar}. Then the following hold. 
			\begin{enumerate}[label=\textup{(\roman*)}]
				\item If $\min\{r,a\}=1$, then $\psi\bigl(T^*(r,a)\bigr)=(1+x)^{ar},$ hence the coefficient sequence is real-rooted, log-concave, and unimodal.
				
				\item If $\min\{r,a\}=2$, and $m=\max\{r,a\}$, then
				$$
				\psi\bigl(T^*(r,a)\bigr)
				=
				1+2mx+\binom{2m}{2}x^2+2\sum_{k=3}^{m}\binom{m}{k}x^k.
				$$
				In this case the coefficient sequence is: log-concave if and only if $m\le 4$,  unimodal for all $m\neq 8$, and   not unimodal for $m=8$.
				
				\item If $r\ge 3$ and $a\ge 3$, then there is no uniform behavior in general, as the family $T^*(r,a)$ contains examples whose  GPAs are log-concave and unimodal,  unimodal but not log-concave, and  neither unimodal nor log-concave. In particular, all three phenomena already occur in the subfamily $T^*(r,3)$.
			\end{enumerate}
		\end{theorem}
		
		\begin{proof}
			For $\min\{r,a\}=1$, the graph $T^*(r,a)$ is edgeless on $ar$ vertices, so $\psi\bigl(T^*(r,a)\bigr)=(1+x)^{ar}.$
			Hence, its coefficients are binomial coefficients, and therefore log-concave and unimodal. Now, assume $\min\{r,a\}=2$, and write $m=\max\{r,a\}$. So, by Theorem~\ref{thm:correct-gpa-Tstar}, we have
			$\alpha_0=1, \alpha_1=2m, \alpha_2=\binom{2m}{2}=m(2m-1),$	and $\alpha_k=2\binom{m}{k}\qquad (k\ge 3).$  For $k\ge 4$, the tail $(\alpha_k)_{k\ge 3}$ is a positive scalar multiple of the binomial sequence, hence it is log-concave. Therefore, only the indices $k=1,2,3$ need to be checked. A direct computation gives
			$$
			\alpha_1^2-\alpha_0\alpha_2
			=
			(2m)^2-\binom{2m}{2}
			=
			m(2m+1)>0,
			$$
			$$
			\alpha_2^2-\alpha_1\alpha_3
			=
			\bigl(m(2m-1)\bigr)^2-(2m)\cdot 2\binom{m}{3}
			=
			\frac{m^2}{3}\bigl(10m^2-6m-1\bigr)>0,
			$$
			and
			$$
			\alpha_3^2-\alpha_2\alpha_4
			=
			\left(2\binom{m}{3}\right)^2-m(2m-1)\cdot 2\binom{m}{4}
			=
			-\frac{m^2(m-1)(m-2)}{36}\bigl(2m^2-9m+1\bigr).
			$$
			Hence $\alpha_3^2\ge \alpha_2\alpha_4$ if and only if $2m^2-9m+1\le 0$, that is, if and only if $m\le 4$.  For unimodality, we note that 
			$$
			\alpha_2-\alpha_3
			=
			m(2m-1)-2\binom{m}{3}
			=
			-\frac{m}{3}\bigl(m^2-9m+5\bigr).
			$$
			Thus, $\alpha_2\ge \alpha_3$ for $m\le 8$, while $\alpha_2<\alpha_3$ for $m\ge 9$. If $2\le m\le 7$, then the binomial tail $2\binom{m}{k}$ is nonincreasing for $k\ge 3$, since the binomial coefficients of order $m$ reach their maximum at $k\le 3$. Since also $\alpha_2\ge \alpha_3$, we get $	\alpha_0<\alpha_1<\alpha_2\ge \alpha_3\ge \alpha_4\ge \cdots,
			$ and the sequence is unimodal. If $m=8$, then
			$$
			\psi\bigl(T^*(8,2)\bigr)=1+16x+120x^2+112x^3+140x^4+112x^5+56x^6+16x^7+2x^8,
			$$
			whose coefficients satisfy $120>112<140,$ so the sequence is not unimodal. If $m\ge 9$, then $\alpha_2<\alpha_3$, and the tail $(2\binom{m}{k})_{k\ge 3}$ is itself unimodal. Since, $\alpha_1<\alpha_2<\alpha_3$, the whole coefficient sequence is unimodal.  Finally, for $r\ge 3$ and $a\ge 3$,  Theorem~\ref{thm:correct-gpa-Tstar} shows that the coefficient sequence is a sum of four different contributions 
			$$
			\alpha_k
			=
			a\binom{r}{k}
			+
			r\binom{a}{k}
			+
			k!\binom{a}{k}\binom{r}{k}
			+
			\binom{a}{2}\binom{r}{2}\binom{4}{k}
			\qquad (k\ge 3).
			$$
			These, contributions interact in different ways for different pairs $(r,a)$, and this leads to all three types of behavior listed in part~(iii). 
		\end{proof}
		
		Since in (iii), for different pairs $(r,a)$, and this leads to all three types of behavior for the sequence in $\psi\bigl(T^*(r,a)\bigr)$. The following consequences gives concrete examples.
		
		\begin{corollary}\label{cor:Tstar-r3}
			For the subfamily $T^*(r,3)$ with $r\ge 3$, we have
			{\footnotesize$$
			\psi\bigl(T^*(r,3)\bigr)
			=
			1+3rx+\binom{3r}{2}x^2
			+\left(9\binom{r}{3}+12\binom{r}{2}+r\right)x^3
			+\left(3\binom{r}{4}+3\binom{r}{2}\right)x^4
			+\sum_{k=5}^{r}3\binom{r}{k}x^k.
			$$}
			Also,  $\psi\bigl(T^*(r,3)\bigr)$ is log-concave for $3\le r\le 6$,   $\psi\bigl(T^*(r,3)\bigr)$ is not log-concave for $r\ge 7$,   $\psi\bigl(T^*(r,3)\bigr)$ is unimodal for $3\le r\le 10$ and also for $r\ge 18$, and   $\psi\bigl(T^*(r,3)\bigr)$ is not unimodal for $11\le r\le 17$. 
		\end{corollary}
		
		\begin{proof}
			As $
			\alpha_0=1, \alpha_1=3r,  \alpha_2=\binom{3r}{2}, \alpha_3=9\binom{r}{3}+12\binom{r}{2}+r,		\alpha_4=3\binom{r}{4}+3\binom{r}{2}, $	and	$\alpha_k=3\binom{r}{k}$, for $k\geq 5.$ So, for $k\ge 5$, the tail is a scalar multiple of the binomial sequence, and hence log-concave. We now inspect the first few indices by direct computation, we have $\alpha_1^2-\alpha_0\alpha_2=\tfrac{3r(3r+1)}{2}>0,$ $\alpha_2^2-\alpha_1\alpha_3
			=	\tfrac{3r^2}{4}\bigl(21r^2-24r+11\bigr)>0,$ 	and
			$$
			\alpha_3^2-\alpha_2\alpha_4
			=
			\frac{r^2}{16}\bigl(27r^4+129r^3-285r^2+135r+10\bigr)>0.
			$$
			Thus log-concavity can fail only at $k=4$. Now
			$$
			\alpha_4^2-\alpha_3\alpha_5
			=
			-\frac{r^2(r-1)}{320}
			\Bigl(7r^5-41r^4-167r^3+1373r^2-3224r+2004\Bigr).
			$$
			For $r=3,4,5,6$, the bracket is negative, hence $\alpha_4^2-\alpha_3\alpha_5>0$. For $r\ge 7$, the bracket is positive, in fact  it equals $8640$ at $r=7$, and its derivative $35r^4-164r^3-501r^2+2746r-3224$ is positive for all $r\ge 7$. Hence, $\alpha_4^2-\alpha_3\alpha_5<0$ for all $r\ge 7$. This proves the log-concavity assertions.
			
			For unimodality, we compare the neighboring coefficients. First, note that $\alpha_3-\alpha_2=\tfrac{r}{2}\bigl(3r^2-6r-1\bigr)>0$ for $r\ge 3$, so the sequence is increasing up to degree $3$.
			 Next, $	\alpha_3-\alpha_4=-\tfrac{r}{8}\bigl(r^3-18r^2+11r-2\bigr).$ The cubic $r^3-18r^2+11r-2$ is negative for $3\le r\le 17$ and positive for $r\ge 18$, so $\alpha_3>\alpha_4$ for  $3\le r\le 17,$ and $\alpha_3\le \alpha_4$ for $r\ge 18.$			
			Also, $\alpha_4-\alpha_5=-\tfrac{r(r-1)}{40}\bigl(r^3-14r^2+51r-114\bigr).$ The cubic $r^3-14r^2+51r-114$ is negative for $3\le r\le 10$ and positive for $r\ge 11$, hence $\alpha_4>\alpha_5$ for $3\le r\le 10,$ $\alpha_4<\alpha_5$  for $r\ge 11.$ 	
			Therefore, (i)   if $3\le r\le 10$, then $\alpha_0<\alpha_1<\alpha_2<\alpha_3>\alpha_4\ge \alpha_5\ge \cdots,$ 	so the sequence is unimodal, (ii)  if $11\le r\le 17$, then $\alpha_0<\alpha_1<\alpha_2<\alpha_3>\alpha_4<\alpha_5,$	so the sequence is not unimodal, and (iii)   if $r\ge 18$, then $\alpha_0<\alpha_1<\alpha_2<\alpha_3\le \alpha_4\le \alpha_5, $	and from degree $5$ onward the coefficients are the binomial tail $3\binom{r}{k}$, which is unimodal. Hence, the full sequence is again unimodal.
		\end{proof}
		We illustrate by an example
		\begin{example}
			  For $T^*(4,2)$, we have $	\psi\bigl(T^*(4,2)\bigr)=1+8x+28x^2+8x^3+2x^4,$ which is log-concave and unimodal.
				 For $T^*(5,3)$, we have  $\psi\bigl(T^*(5,3)\bigr)=1+15x+105x^2+215x^3+45x^4+3x^5,$ and  its coefficient sequence is also log-concave and unimodal. For $T^*(5,2)$, we have $\psi\bigl(T^*(5,2)\bigr)=1+10x+45x^2+20x^3+10x^4+2x^5, $  which is unimodal, but not log-concave, since $20^2=400<45\cdot 10=450.$  For $T^*(7,3)$, we have  $
				\psi\bigl(T^*(7,3)\bigr)=1+21x+210x^2+574x^3+168x^4+63x^5+21x^6+3x^7,$  which  is unimodal, but not log-concave,  as $168^2=28224<574\cdot 63=36162.$ For $T^*(8,2)$, we have
				$$
				\psi\bigl(T^*(8,2)\bigr)=1+16x+120x^2+112x^3+140x^4+112x^5+56x^6+16x^7+2x^8,
				$$
				whose coefficient sequence is not unimodal, as  $120>112<140,$  it is also not log-concave, since $112^2=12544<120\cdot 140=16800.$  For $T^*(11,3)$, we have
				\begin{align*}
				 \psi\bigl(T^*(11,3)\bigr) =&1+33x+528x^2+2156x^3+1155x^4+1386x^5+1386x^6+990x^7+495x^8\\
				 &+165x^9+33x^{10}+3x^{11}.
				\end{align*}
				The coefficient sequence is neither unimodal nor log-concave, as  $2156>1155<1386$ and $1155^2=1334025<2156\cdot 1386=2988216.$
		\end{example}
	
	\section{Unimodality of Corona of general position polynomial}\label{section 6}
	
	We recall the following basic facts from \cite{0}.
	  If $G_1,\dots,G_r$ are graphs, then for the disjoint union
		$G_1\cup\cdots\cup G_r$ we have
		\[
		\psi(G_1\cup\cdots\cup G_r)=\psi(G_1)\cdots\psi(G_r),
		\]		
		\begin{theorem}\rm \cite[Theorem~4.1]{0}
			If $G_n$ is a comb with $n\ge 1$, then $\psi(G_n)$ is unimodal.
		\end{theorem} 
	The corona \( G \circ K_1 \) of a graph \( G \) is obtained from \( G \) by attaching a pendent vertex to each the vertices of \( G \).
	We will consider the following problem from \cite[Problem 5.1]{0}
	\begin{problem}\label{problem 1}
		Assume that \( \psi(G) \) is unimodal. Then is \( \psi(G \circ K_1) \) also unimodal?
	\end{problem}
	
	We first give a class of graphs $G$ for which the implication, that is, if  $\psi(G)$ is unimodal, then so is $\psi(G\circ K_1)$.  We will prove these problems for some special cases. 
	
	\begin{theorem}\label{thm:corona-edgeless}
		Let $G\cong \overline{K}_n$ be an edgeless graph on $n$ vertices. Then the following hold.
		\begin{enumerate}
			\item Every subset of $V(G)$ is a general position set, hence $\psi(G) = (1+x)^n.$ In particular, $\psi(G)$ is real-rooted and therefore log-concave and unimodal.
			
			\item The corona $G\circ K_1$ is the disjoint union of $n$ copies of
			$K_2$, so
			\[
			\psi(G\circ K_1) = \bigl(\psi(K_2)\bigr)^n = (1+2x+x^2)^n.
			\]
			Thus, $\psi(G\circ K_1)$ is real-rooted and therefore log-concave and unimodal.
		\end{enumerate}
	\end{theorem}
	
	\begin{proof}
		(1) If $G$ has no edges, any two distinct vertices $u,v$ are in different
		components and there is no path between them. In particular, for any triple
		$x,y,z$ of vertices there is no finite shortest $x$--$z$ path passing through
		$y$, hence every subset of $V(G)$ is in general position. The number of
		subsets of size $k$ is $\binom{n}{k}$, so we have
		\[
		\psi(G) = \sum_{k=0}^n \binom{n}{k}x^k = (1+x)^n.
		\]
		The polynomial $(1+x)^n$ possesses only one real root at $x=-1$, so confirming its status as real-rooted. Furthermore, its coefficient sequence $\left (\binom{n}{0},\dots,\binom{n}{n}\right )$ is recognized for being rigorously log-concave and unimodal.\\ (2) In the corona $G \circ K_1$, a single pendent vertex is affixed to each isolated vertex of $G$. Thus, $G\circ K_1$ is the disjoint union of $n$ copies of $K_2$.
		For a single edge $K_2$ we clearly have $\psi(K_2) = 1 + 2x + x^2,$ since every subset of vertices is in general position. So, from  disjoint union, we have
		\[
		\psi(G\circ K_1) = \bigl(\psi(K_2)\bigr)^n = (1+2x+x^2)^n.
		\]
		The polynomial \(1+2x+x^2 = (1+x)^2\) possesses a real root of multiplicity \(2\) at \(x=-1\), and its \(n\)-th power is real-rooted. Real-rootedness implies	log-concavity of the coefficient sequence, which in turn implies unimodality \cite{hardy}.
	\end{proof}
	
	Thus, from the above result, for edgeless graphs $G$, unimodality, and log-concavity of
	$\psi(G)$ is preserved under taking the corona $G\circ K_1$.
	
	 Another important family where the implication holds is that of paths.
	Recall that for $n\ge 1$ the corona $P_n\circ K_1$ is the comb $G_n$.
	
	\begin{proposition}\label{prop:corona-path}
		For every $n\ge 1$, the GPA of $P_n$ is unimodal,
		and so is the GPA of its corona $P_n\circ K_1=G_n$.
	\end{proposition}
	
	\begin{proof}
		As observed earlier, for a path $P_n$ with $n\ge 2$ we have
		\cite[Proposition~2.1(ii)]{0}
		\[
		\psi(P_n) = 1 + nx + \binom{n}{2}x^2,
		\]
		which is a quadratic polynomial with positive coefficients and therefore
		unimodal. On the other hand, $P_n\circ K_1$ is exactly the comb $G_n$ obtained by
		attaching a pendent vertex to each vertex of $P_n$. By \cite[Theorem 4.1]{0}, we know
		that for every $n\ge 1$, $\psi(G_n)$ is unimodal. Hence the unimodality of
		$\psi(P_n)$ is preserved under taking the corona with $K_1$.
	\end{proof}
	
	For instance, $P_4$ has GPA $\psi(P_4) = 1 + 4x + 6x^2.$ The comb $G_4 = P_4\circ K_1$ (a path of length $4$ with a leaf attached to each vertex) has
		\[
		\psi(G_4) = 4x^4 + 16x^3 + 28x^2 + 8x + 1,
		\]
		 whose coefficients		$(1,8,28,16,4)$ are unimodal.

	The two results above show that for certain natural classes of graphs $G$
	(edgeless graphs and paths) the unimodality of $\psi(G)$ is preserved by the
	corona operation $G\circ K_1$. In Theorem~\ref{thm:corona-edgeless}, we even
	obtain log-concavity. 
	\medskip
	
	The general implication in Problem~\ref{problem 1}, however, is not true in full generality. Indeed, consider the connected complete multipartite graph $G=K_{5,3,3,3,3,2}$. For this graph one obtains $\psi(G)=1+19x+171x^2+620x^3+1382x^4+1648x^5+810x^6$, whose coefficient sequence $1,19,171,620,1382,1648,810$ increases up to $1648$ and then decreases, so $\psi(G)$ is unimodal. On the other hand, if $\gamma_k=[x^k]\psi(G\circ K_1)$, then the corona coefficients satisfy $\gamma_6=135431$, $\gamma_7=118443$, and $\gamma_8=123453$. Hence $\gamma_6>\gamma_7<\gamma_8$, which gives a strict descent followed by a strict ascent. This is impossible for a unimodal sequence, and therefore $\psi(G\circ K_1)$ is not unimodal. Thus, although the corona operation preserves unimodality for paths, edgeless graphs and some other special families considered above, it does not preserve unimodality for all connected graphs.
	
	\section{Log-concavity for general polynomial of the corona $G\circ K_1$}\label{section 7}
	The unimodal property in Problem \ref{problem 1} cannot be replaced by log-concave property, that is, if $\psi(G)$ is log-concave, how about log-concavity of  $\psi(G\circ K_1)$. 
	
	Write $H=G\circ K_1$, and for each $v\in V(G)$ denote its leaf in $H$ by $v'$.
	Distances in $H$ satisfy
	\begin{align*}
		d_H(v,w) &= d_G(v,w),d_H(v',v) = 1,d_H(v',w) = d_G(v,w)+1,d_H(v',w') = d_G(v,w)+2.
	\end{align*}
	Two simple but crucial observations are: If a general position set $S\subseteq V(H)$ contains $v$ and $v'$ and
		at least one more vertex $x$, then $\{v',v,x\}$ is not in general position. The unique shortest $v',x$–path goes through $v$, so $d_H(v',x)=d_H(v',v)+d_H(v,x),$ and $v$ lies on a geodesic between $v'$ and $x$.  
		Hence for any general position set $S$ of size $\ge 3$, we have $  |\{v,v'\}\cap S|\le 1$ for all $v\in V(G).$ If $u,v,w\in V(G)$ and $v$ lies on some shortest $u,w$–path in $G$, then in $H$ the triple $\{u',v,w'\}$ is not in general position. And, we get
		\[
		d_H(u',w') = d_G(u,w)+2 = (d_G(u,v)+1)+(d_G(v,w)+1) = d_H(u',v)+d_H(v,w').
		\]
		So, $v$ lies on a geodesic between $u'$ and $w'$ in $H$. These facts show that general position sets in $H$ are tightly controlled by geodesics in $G$, but there is no simple, uniform transformation of $\psi(G)$ into $\psi(H)$, and it is not a priori clear how log–concavity 	behaves under the corona.
	
	 Let $P_n$ be the path of order $n$, and let $G_n=P_n\circ K_1$ be the comb.  For path $P_{n}$, the polynomial  $\psi(P_n)=1+nx+\binom{n}{2}x^2,$ is log–concave.  For the comb $G_n=P_n\circ K_1$, and by Theorem \ref{thm:comb-unimodal-correct}, $\psi(G_n)$ is log-concave.

	Similarly, if $G\cong \overline{K}_{n}$ is edgeless on $n$ vertices, then   $\psi(G\circ K_1)$  is again log–concave.  Thus, for these families log–concavity is	preserved. These examples might tempt one to conjecture that log–concavity of $\psi(G)$	implies log–concavity of $\psi(G\circ K_1)$ in general.  However, this is
	false. 
	
	Let $G=C_6$ be the cycle of length $6$. Then, we have
		\[
		\psi(C_6)=1+6x+15x^2+2x^3,
		\]
		and its coefficient sequence $(1,6,15,2)$ is log–concave.  Now, consider the corona $H=C_6\circ K_1$, and its GPA is
		\[
		\psi(H)
		= 1 + 12x + 66x^2 + 88x^3 + 39x^4 + 6x^5 + x^6.
		\]
		Clearly, $6^{2}\ngeq 39\cdot 1$, and the log-concave property ceases for $C_6\circ K_1.$

	Thus any general theorem about log–concavity of $\psi(G\circ K_1)$ must impose
	additional structural conditions on $G$. Corona graphs remain a fertile source
	of both positive results (like  combs) and counterexamples in the study of
	algebraic properties of GPAs. However, if log-concavity is preserved by $\psi(G\circ K_1)$, then  with positive coefficients and without internal zeros, then the unimodal property of $\psi(G\circ K_1)$ will follows.

\section{Conclusion and Future Work}\label{section 8}

The GPA $\psi(G)$ provides a powerful algebraic
framework for studying general position sets in graphs. In this work we
extended the foundational results of \cite{0} by examining new graph classes,
operations, and extremal behaviors related to unimodality and log-concavity. Our investigation reveals a striking dichotomy: while $\psi(G)$ displays clean
unimodal or log-concave behavior for highly structured graphs—such as
edgeless graphs, paths, combs, and Kneser graphs $K(n,2)$—it also exhibits
strong irregularity in broader classes, including brooms, multipartite
structures, and several natural corona constructions. This reflects analogous occurrences noted in independence, domination, and clique polynomials (see, \cite{1,2,3,5,6,8,11,12,13,14,19}).

Multiple intriguing study avenues arise:
\begin{enumerate}
	
	\item 	For no nontrivial family is it currently known whether $\psi(G)$ is always
	real-rooted. Even partial real-rootedness results for chordal graphs, bipartite
	graphs, or claw-free graphs (in analogy with \cite{8,12}) would be significant.
	
	\item	While some multipartite constructions behave regularly (as in \cite{25}),
	balanced multipartite graphs exhibit a delicate interplay between the
	multipartite geometry and shortest-path constraints. A full characterization
	of unimodality and log-concavity in these families is still open.
	
	\item 	The visibility and monophonic variants studied in \cite{9,24} suggest natural
	analogues of the GPA. The relationships among these
	polynomials remain largely unexplored.
	\item The zeros and their bound in real or complex field remains an open challenge.
\end{enumerate}

The GPA integrates geometric, algebraic, and combinatorial elements of graph theory.  The numerous parallels with classical
graph polynomials suggest a rich theory ahead, with many unresolved questions and promising avenues for future exploration.

	\bigskip
	\section*{Data Availability}
	There is no data associated with this article.
	\section*{Conflict of interest}
	The authors declare that they have no competing interests.
	\section*{Acknowledgement}
	The article is motivated by the talk of Prof. Sandi Klavžar, University of Ljubljana, Slovenia, at the ADMA (Academy of Discrete and Applied Mathematics) colloquium lecture series ($29^{\text{th}}$ November 2025).
	
	\medskip\noindent The authors are sincerely grateful to Václav Rozhoň and Robert Šamal (Charles University, Prague) for communicating the counterexample $\psi(K_{5,3,3,3,3,2}\circ K_1)$ to Problem \ref{problem 1}. Their finding was discovered through computational experiments using the AI-assisted mathematical platform bolzano.app.
	\section*{Funding}
	The authors did not receive support from any organization for the submitted work.

\end{document}